\documentclass[a4paper,11pt,reqno,oneside]{amsart}
\usepackage{amsmath,amssymb,amsthm,mathtools}
\usepackage{enumitem,xcolor}
\usepackage[scale=0.75,hcentering]{geometry}
\usepackage{cite}
\usepackage{xurl}
\usepackage{microtype,aliascnt,hyperref}
\usepackage[capitalize,nameinlink,noabbrev]{cleveref}
\makeatletter
\@ifundefined{subjclassname@2020}{\@namedef{subjclassname@2020}{%
\textup{2020} Mathematics Subject Classification}}{}
\makeatother
\definecolor{cobalt}{rgb}{0.0,0.28,0.67}
\hypersetup{bookmarksopen=true,
pdftitle={Rank and Border Rank Additivity of Complex Tensors with a Hankel Summand}, 
pdfauthor={Jianze Li},
pdfsubject={Exact and border rank additivity for complex Hankel tensors},
pdfkeywords={Hankel tensor, CP rank, border rank, direct sum, apolarity},
colorlinks=true,linkcolor=cobalt,citecolor=cobalt,urlcolor=cobalt}
\theoremstyle{plain}
\newtheorem{theorem}{Theorem}[section]
\newaliascnt{proposition}{theorem}
\newtheorem{proposition}[proposition]{Proposition}
\aliascntresetthe{proposition}
\newaliascnt{problem}{theorem}
\newtheorem{problem}[problem]{Open Problem}
\aliascntresetthe{problem}
\newaliascnt{lemma}{theorem}
\newtheorem{lemma}[lemma]{Lemma}
\aliascntresetthe{lemma}
\newaliascnt{corollary}{theorem}
\newtheorem{corollary}[corollary]{Corollary}
\aliascntresetthe{corollary}
\theoremstyle{remark}
\newaliascnt{remark}{theorem}
\newtheorem{remark}[remark]{Remark}
\aliascntresetthe{remark}
\newaliascnt{example}{theorem}
\newtheorem{example}[example]{Example}
\aliascntresetthe{example}
\crefname{theorem}{Theorem}{Theorems}
\crefname{proposition}{Proposition}{Propositions}
\crefname{lemma}{Lemma}{Lemmas}
\crefname{corollary}{Corollary}{Corollaries}
\crefname{remark}{Remark}{Remarks}
\crefname{example}{Example}{Examples}
\crefname{problem}{Open Problem}{Open Problems}
\newcommand{\tens}[1]{\boldsymbol{#1}}
\newcommand{\matr}[1]{\mathbf{#1}}
\newcommand{\vect}[1]{\boldsymbol{#1}}
\newcommand{\set}[1]{\mathcal{#1}}
\newcommand{\eqdef}{\mathrel{\triangleq}}
\newcommand{\CC}{\mathbb{C}}
\newcommand{\br}{\underline r}
\newcommand{\bs}{\underline s}
\DeclareMathOperator{\Sym}{\mathsf{Sym}}
\DeclareMathOperator{\Span}{\mathsf{span}}
\DeclareMathOperator{\rank}{\mathsf{rank}}
\DeclareMathOperator{\codim}{\mathsf{codim}}
\DeclareMathOperator{\Stab}{\mathsf{Stab}}
\numberwithin{equation}{section}
\allowdisplaybreaks[2]
\setlist[enumerate]{label=\textup{(\roman*)},leftmargin=*}
\title[Rank and Border Rank Additivity of Complex Tensors with a Hankel Summand]
{Rank and Border Rank Additivity of Complex Tensors with a Hankel Summand}
\author{Jianze Li}
\address{School of Science, Sun Yat-sen University, Shenzhen, Guangdong, China}
\email{lijianze@mail.sysu.edu.cn}
\date{\today}
\subjclass[2020]{15A69, 14N07, 13H10, 68Q17}
\keywords{Hankel tensor, CP rank, symmetric rank, direct sum,
border rank, apolarity, Frobenius algebra, Koszul flattening}

\begin{document}
\begin{abstract}
We prove that the CP rank of a complex Hankel tensor of order
$m\ge3$ is additive with that of every complex tensor of the same
order. Together with the known equality of CP and symmetric ranks
for complex Hankel tensors, this yields symmetric rank additivity
whenever the other summand has equal CP and symmetric ranks.
We also prove that ordinary and symmetric border ranks coincide
and are additive on direct sums of complex Hankel tensors
of any common order at least three. Their common value is the sum
of the middle binary catalecticant ranks. 
For order three, we further prove border-rank additivity when the
other summand has border rank at most three, or is symmetric with
at most five essential variables. 
\end{abstract}

\maketitle
\enlargethispage{3pt}

\section{Introduction}\label{sec:intro}

\subsection{Tensor CP rank, symmetric rank, and direct sums}
Tensor CP rank is the minimum number of simple tensors needed to
represent a tensor \cite{hitchcock1927expression,kruskal1977three,kolda2009tensor,landsberg2012tensors}. 
For symmetric tensors, requiring the simple summands to be symmetric
gives the symmetric rank, or equivalently the Waring rank of the
associated homogeneous polynomial
\cite{comon2008symmetric,landsberg2012tensors}.
In this paper, all vector spaces are finite-dimensional over
$\CC$, and all tensor ranks are computed over $\CC$.

In all rank definitions below, the minimum is taken over
$q\in\mathbb Z_{\ge0}$.
Let $m\ge2$ and let $\set{A}_1,\ldots,\set{A}_m$ be finite-dimensional vector spaces.
For a tensor $\tens{T}\in\bigotimes_{i=1}^m\set{A}_i$, its CP rank is
\[
r(\tens{T})\eqdef
\min\left\{q:\tens{T}=\sum_{h=1}^q
\vect{a}_{h,1}\otimes\cdots\otimes\vect{a}_{h,m},\
\vect{a}_{h,i}\in\set{A}_i\right\}.
\]
For tensors $\tens{T}_1\in\bigotimes_{i=1}^m\set{A}_i$ and $\tens{T}_2\in\bigotimes_{i=1}^m\set{B}_i$,
the direct sum $\tens{T}_1\oplus\tens{T}_2$ belongs to
$\bigotimes_{i=1}^m(\set{A}_i\oplus\set{B}_i)$.
Its two diagonal blocks are $\tens{T}_1$ and $\tens{T}_2$, and all
mixed blocks vanish. Concatenating minimum CP decompositions gives
$r(\tens{T}_1\oplus\tens{T}_2)\le r(\tens{T}_1)+r(\tens{T}_2)$.
For a symmetric tensor $\tens{S}\in\Sym^m(\set{A})$, its symmetric rank is
\[
s(\tens{S})\eqdef
\min\left\{q:\tens{S}=\sum_{h=1}^q
\lambda_h\vect{a}_h^{\otimes m},\
\lambda_h\in\CC,\ \vect{a}_h\in\set{A}\right\}.
\]
Thus $r(\tens{S})\le s(\tens{S})$, and symmetric rank is also
subadditive under direct sums. 
We define the Strassen CP rank gap, Strassen symmetric rank gap,
and Comon rank gap by
\begin{align*}
\Delta_r(\tens{T}_1,\tens{T}_2)
&\eqdef r(\tens{T}_1)+r(\tens{T}_2)-r(\tens{T}_1\oplus\tens{T}_2),\\
\Delta_s(\tens{S}_1,\tens{S}_2)
&\eqdef s(\tens{S}_1)+s(\tens{S}_2)-s(\tens{S}_1\oplus\tens{S}_2),\ \
\Delta_c(\tens{S})\eqdef s(\tens{S})-r(\tens{S}).
\end{align*}
The latter two definitions concern symmetric tensors.
All three gaps are nonnegative. In particular, a Strassen rank gap
vanishes exactly when the corresponding rank is additive.
The Comon rank gap is also defined in
\cite{li2026comonrankgap} for real symmetric third-order tensors.
Here we use the analogous definition over $\CC$ for symmetric tensors of arbitrary order.

We now recall Strassen's direct sum conjecture. In the setting
of bilinear complexity, Strassen asked in 1973 whether CP rank
of third-order tensors is always additive under direct sums
\cite{strassen1973division}. 
Feig and Winograd studied the stronger
requirement that every minimum decomposition separate into the two
blocks \cite{feig1984directsum}; see also
\cite{shitov2019directsum}.
For tensors of order three, positive cases include a summand with
an essential mode dimension at most two
\cite{jaja1986validity}, as well as several small formats and ranks
\cite{buczynski2020additivity,rupniewski2024additivity}.
However, Shitov disproved the general conjecture over every infinite
field \cite{shitov2019directsum} in 2019.
As a symmetric analogue, Strassen's symmetric direct sum conjecture asks for additivity of symmetric rank
\cite{carlini2015symmetricstrassen}.
Known positive cases include sums of monomials in disjoint variables
\cite{carlini2012monomials}, binary forms and further families
\cite{carlini2015symmetricstrassen,carlini2018ecomputability},
and criteria based on lower bounds or Waring loci
\cite{teitler2015sufficient,carlini2017waringloci}.
Shitov's preprint \cite[Theorem~3.9]{shitov2023higher} reports counterexamples
to this version as well. See \cite{casarotti2020comon} for a survey
of the earlier developments. 

\subsection{Tensor border rank and symmetric border rank}
The sets of tensors of bounded CP rank need not be closed; consequently,
a tensor may be approximated arbitrarily closely by tensors of
strictly smaller rank \cite{de2008tensor,landsberg2012tensors}.
Border rank captures this limiting complexity and underlies
approximate bilinear algorithms and their applications to fast matrix
multiplication \cite{bini1979approximate,bini1980approximate}.
The border rank and symmetric border rank are defined by
\begin{align*}
\br(\tens{T})&\eqdef\min\{q:\tens{T}\in
   \overline{\{\tens{T}':r(\tens{T}')\le q\}}\},\\
\bs(\tens{S})&\eqdef\min\{q:\tens{S}\in
   \overline{\{\tens{S}'\in\Sym^m(\set{A}):s(\tens{S}')\le q\}}\}.
\end{align*}
The second definition assumes that $\tens{S}$ is symmetric.
The closures are Zariski closures in the corresponding complex
tensor spaces; equivalently, one may use limits in the Euclidean
topology. In the first definition the approximating tensors are
arbitrary, even when $\tens{T}$ is symmetric or Hankel.
Thus $\br(\tens{S})\le r(\tens{S})\le s(\tens{S})$ and $\br(\tens{S})\le\bs(\tens{S})\le s(\tens{S})$ whenever the symmetric
quantities are defined. 
Both border ranks are subadditive under direct sums. 
We similarly define their three
corresponding gaps as 
\begin{align*}
\underline{\Delta}_r(\tens{T}_1,\tens{T}_2)
&\eqdef\br(\tens{T}_1)+\br(\tens{T}_2)-\br(\tens{T}_1\oplus\tens{T}_2),\\
\underline{\Delta}_s(\tens{S}_1,\tens{S}_2)
&\eqdef\bs(\tens{S}_1)+\bs(\tens{S}_2)-\bs(\tens{S}_1\oplus\tens{S}_2),\ \
\underline{\Delta}_c(\tens{S})\eqdef\bs(\tens{S})-\br(\tens{S}).
\end{align*}
These gaps are nonnegative, and the latter two definitions concern symmetric tensors.

The border-rank analogue of Strassen's direct sum conjecture asks whether
$\underline{\Delta}_r(\tens{T}_1,\tens{T}_2)=0$ always holds.
In early work on approximate matrix multiplication,
Sch\"onhage disproved this assertion for third-order tensors
\cite{schonhage1981partial}; his construction applies over $\CC$. Christandl, Gesmundo, Micha\l ek, and
Zuiddam gave complex counterexamples in every order at least four
\cite{christandl2021border}.
Positive results include complex third-order direct sums whose total mode dimensions are at most four
\cite{buczynski2020additivity}.
The symmetric border-rank question asks instead whether
$\underline{\Delta}_s(\tens{S}_1,\tens{S}_2)=0$ for forms in disjoint
sets of variables over $\CC$.
Teitler proves additivity when a common catalecticant degree
attains each summand's symmetric border rank
\cite[Example~3.6]{teitler2015sufficient}; binary forms are included
by their middle catalecticants
\cite[Theorem~2.2]{nie2019hankel}.
To our knowledge, the general symmetric border-rank
assertion over $\CC$ remains open.

\subsection{Hankel tensors and related work}
A tensor $\tens{H}\in(\CC^n)^{\otimes m}$ is Hankel \cite{luque2003hankel,qi2015hankel,nie2019hankel} if there exist
complex numbers $h_0,\ldots,h_{m(n-1)}$ such that
$H_{i_1\ldots i_m}=h_{i_1+\cdots+i_m}$ for
$0\le i_1,\ldots,i_m<n$.
The sequence $(h_0,\ldots,h_{m(n-1)})$ is called the
generating sequence of $\tens{H}$.
Every Hankel tensor is symmetric; we denote their space by
$\set{H}_m(n)$.
Early work studied Hankel hyperdeterminants and moment integrals
\cite{luque2003hankel}.
Structural developments include strong and complete Hankel tensors,
Vandermonde decompositions, and positivity over $\mathbb R$
\cite{qi2015hankel,xu2016hankel}, as well as inheritance properties
and sum-of-squares decompositions \cite{ding2017inheritance}.
Applications and algorithms include exponential data fitting
and fast Hankel tensor--vector products
\cite{papy2005exponential,ding2015fast}.
Related Hankel-matrix and moment-algebra methods connect tensor
decomposition with Artinian Gorenstein algebras and
polynomial--exponential representations
\cite{brachat2010symmetric,mourrain2018moments}.
Recent work addresses decomposition on rational varieties,
tensor completions for cactus rank, and multiplication-operator
methods for polynomial roots and tensor decomposition
\cite{bechere2026rational,bernardi2026hankel,barrilli2026eigensolvers}.

For CP rank and symmetric rank over $\CC$, Nie and Ye gave an algorithm
for computing the Vandermonde rank
\cite{nie2019hankel}, 
and proved rank equality for generic Hankel tensors of even
order and of order three
\cite{nie2019hankel},
and formulated Comon's conjecture for all complex Hankel tensors
\cite[Conjecture~6.3]{nie2019hankel}.
More recently, 
Colbrook's repository manuscript extends generic rank equality
to the remaining odd orders
\cite{colbrook2026generic}.
Holden's repository manuscript gives an exact formula and proves
that $\Delta_c(\tens{H})=0$ for every complex Hankel tensor
\cite{holden2026exact}. 
For border rank, Nie and Ye prove the single-tensor equalities
in even order and in odd order
$m=2q+1$ when the apolar degree $\rho$ satisfies
$\rho\le1+q(n-1)$ \cite{nie2019hankel}.
Holden establishes the single-tensor border-rank equalities
in every order at least three \cite{holden2026exact}.

\subsection{Contributions}
In this paper, we study Strassen's direct sum conjecture with a complex Hankel summand. 
Our main contributions include three parts. First, we prove CP rank
additivity for a complex Hankel tensor and an arbitrary complex
tensor of the same order, for every order at least three.
The proof combines a general criterion
for direct-sum additivity with a uniform lower bound on the
CP rank of a Hankel tensor after restricting one input.
Consequently, symmetric rank is also additive whenever the
partner has equal CP and symmetric ranks; in particular, this
applies to two Hankel summands. 
Second, we prove that ordinary and symmetric border ranks
coincide and are additive for direct sums of complex
Hankel tensors of any common order at least three.
Their common value is the sum of the middle binary catalecticant
ranks of the summands. 
Finally, for third-order tensors, we extend border-rank
additivity beyond Hankel partners in two directions.
The partner may be an arbitrary complex tensor of border rank
at most three, or a symmetric complex tensor with at most five
essential variables. For the latter class, we also establish
equality of ordinary and symmetric border ranks, without an
additional border-rank restriction. Therefore, both border ranks
are additive in the symmetric case.
These extensions reuse the Hankel matrix bounds and the
common-factor construction, supplemented by sharper bounds
for the partner.

\subsection{Organization}
The paper is organized as follows. 
The main results are stated in \cref{sec:main-results}.
\Cref{sec:criterion} gives the common gap identity and a criterion
for CP rank additivity.
\Cref{sec:moment} collects the shared apolar and moment data,
the coefficient pairing, and the known exact-rank formula.
\Cref{sec:partition,sec:restriction} prove the partition and
restriction bounds, and \cref{sec:restriction} concludes with
the proofs of \cref{thm:main,cor:symmetric,cor:pair}.
\Cref{sec:border} proves border additivity using the shared
algebraic data and a joint matrix certificate.
\Cref{sec:cubic-bounds} treats third-order partners of small border
rank and establishes a five-slice Hankel bound.
\Cref{sec:cubic-symmetric} then treats symmetric partners with at
most five essential variables and completes the third main proof.
Finally, \cref{sec:conclusion} gives the finite-sum consequences,
rank differences, and examples, while \cref{sec:outlook} states
the remaining general question.

\section{Main results}\label{sec:main-results}

\subsection{CP and symmetric rank additivity}
The following theorem proves CP rank additivity with every complex
tensor of the same order, for all orders at least three.
\begin{theorem}[CP rank additivity with a Hankel summand]
\label{thm:main}
Let $m\ge3$, let $\tens{H}\in\set{H}_m(n)$, and let
$\tens{T}\in\set{B}_1\otimes\cdots\otimes\set{B}_m$.
Then
\[
r(\tens{H}\oplus\tens{T})=r(\tens{H})+r(\tens{T}).
\]
Equivalently, we have $\Delta_r(\tens{H},\tens{T})=0$.
\end{theorem}

Combining \cref{thm:main} with
$\Delta_c(\tens{H})=0$ from \cite{holden2026exact}
gives the following two consequences.

\begin{corollary}[Symmetric rank additivity]\label{cor:symmetric}
Let $m\ge3$, let $\tens{H}\in\set{H}_m(n)$, and let
$\tens{S}\in\Sym^m(\set{B})$ satisfy $\Delta_c(\tens{S})=0$.
Then
\[
s(\tens{H}\oplus\tens{S})=s(\tens{H})+s(\tens{S}).
\]
Moreover, we have $r(\tens{H}\oplus\tens{S})=s(\tens{H}\oplus\tens{S})$.
Thus $\Delta_r(\tens{H},\tens{S})=\Delta_s(\tens{H},\tens{S})
=\Delta_c(\tens{H}\oplus\tens{S})=0$.
\end{corollary}

\begin{corollary}[Two Hankel summands]\label{cor:pair}
Let $m\ge3$, let $\tens{H}_1\in\set{H}_m(n)$, and let
$\tens{H}_2\in\set{H}_m(p)$. Then
\[
r(\tens{H}_1\oplus\tens{H}_2)=s(\tens{H}_1\oplus\tens{H}_2)
=s(\tens{H}_1)+s(\tens{H}_2).
\]
\end{corollary}

\subsection{Border rank of Hankel sums}

For a nonzero $\tens{H}\in\set{H}_m(n)$, put $D=m(n-1)$
and define its middle binary catalecticant rank
\cite{nie2019hankel} by
\begin{equation}\label{eq:rho-middle}
\rho(\tens{H})\eqdef\rank
 (h_{i+j})_{0\le i\le\lfloor D/2\rfloor,\ 0\le j\le\lceil D/2\rceil}.
\end{equation}
Set $\rho(\tens{0})=0$. For a nonzero scalar tensor, the definition
gives $\rho=1$. We now establish the following theorem.

\begin{theorem}[Border rank of two Hankel summands]\label{thm:border}
Let $m\ge3$, let $\tens{H}_1\in\set{H}_m(n)$, and let
$\tens{H}_2\in\set{H}_m(p)$. Then
\begin{equation}\label{eq:border-finite}
\br(\tens{H}_1\oplus\tens{H}_2)=\bs(\tens{H}_1\oplus\tens{H}_2)
=\rho(\tens{H}_1)+\rho(\tens{H}_2).
\end{equation}
\end{theorem}

Taking one summand to be zero recovers the single-tensor equality
of \cite[Theorem~6.1]{holden2026exact}.
\begin{corollary}[Border ranks of a complex Hankel tensor]
\label{cor:hankel-border-ranks}
Let $m\ge3$ and let $\tens{H}\in\set{H}_m(n)$. Then
\[
\br(\tens{H})=\bs(\tens{H})=\rho(\tens{H}).
\]
\end{corollary}

The symmetric consequence for a more general partner requires
ordinary border-rank additivity as well as equality of that
partner's two border ranks.
\begin{corollary}[Transfer to symmetric border rank]
\label{cor:border-symmetric}
Let $\tens{H}\in\set{H}_m(n)$ and
$\tens{S}\in\Sym^m(\set{B})$ with $m\ge3$. If
$\underline{\Delta}_r(\tens{H},\tens{S})=0$ and
$\underline{\Delta}_c(\tens{S})=0$, then
\[
\br(\tens{H}\oplus\tens{S})=\bs(\tens{H}\oplus\tens{S})
=\rho(\tens{H})+\bs(\tens{S}).
\]
Thus $\underline{\Delta}_s(\tens{H},\tens{S})
=\underline{\Delta}_c(\tens{H}\oplus\tens{S})=0$.
Both hypotheses hold when $\tens{S}$ is itself a Hankel tensor of
order $m$.
\end{corollary}

\subsection{Border rank with third-order partners}\label{subsec:bord_rank_part}
Unlike \cref{thm:main}, \cref{thm:border} requires both summands to be
Hankel. 
The present results do not cover border-rank additivity with an unrestricted partner.
In fact, even the special case
$\br(\tens{T} \oplus 1)=\br(\tens{T})+1,$ 
where $\tens{T}$ is an arbitrary complex third-order tensor and
$1\in\CC\otimes\CC\otimes\CC$, remains open. This rank-one additivity
problem is attributed to Landsberg in
\cite[Problem~2]{buczynski2020additivity}; to our knowledge,
neither a proof nor a counterexample is known.

In this subsection, we restrict our attention to third-order complex tensors
and establish two sufficient conditions for border-rank additivity.
For a symmetric tensor $\tens{S}$, let $d(\tens{S})$ denote its essential
dimension, namely, the dimension of its common mode support or,
equivalently, the rank of any one-mode flattening of $\tens{S}$.

\begin{theorem}[Border rank additivity with a third-order Hankel summand]
\label{thm:cubic-partners}
Let $\tens{H}\in\set{H}_3(n)$ and let $\tens{T}$ be a complex
third-order tensor. Suppose that either
\begin{enumerate}
\item $\br(\tens{T})\le3$; or
\item $\tens{T}$ is symmetric and $d(\tens{T})\le5$.
\end{enumerate}
Then $\br(\tens{H}\oplus\tens{T})=\br(\tens{H})+\br(\tens{T})$.
Moreover, whenever $\tens{T}$ is symmetric under either condition,
one has $\br(\tens{T})=\bs(\tens{T})$ and
$\bs(\tens{H}\oplus\tens{T})=\br(\tens{H}\oplus\tens{T})
=\bs(\tens{H})+\bs(\tens{T})$.
\end{theorem}

% \begin{remark}
% Over $\CC$, Buczy\'nski, Postinghel, and Rupniewski prove
% border rank additivity when each mode of the entire direct sum
% has dimension at most four
% \cite{buczynski2020additivity}.
% In contrast, \cref{thm:cubic-partners} uses the Hankel structure
% of $\tens{H}$ and places no upper bound on its dimension. 
% For symmetric cubic tensors over $\CC$, Seigal proves
% $\br(\tens{S})=\bs(\tens{S})$ when $d(\tens{S})\le4$ and,
% in arbitrary dimension, when $\bs(\tens{S})\le5$
% \cite{seigal2020ranks}.
% Condition~(ii) of \cref{thm:cubic-partners} extends the dimension range to five
% essential variables without a separate border rank assumption.
% It also establishes both direct sum additivity statements
% with a Hankel summand of arbitrary dimension. 
% \end{remark}

\begin{remark}
Over $\CC$, Buczy\'nski, Postinghel, and Rupniewski prove
border rank additivity when each mode of the entire direct sum
has dimension at most four
\cite{buczynski2020additivity}.
In contrast, \cref{thm:cubic-partners} uses the Hankel structure
of $\tens{H}$ and places no upper bound on its dimension.
For symmetric cubic tensors over $\CC$, Seigal proves
$\br(\tens{S})=\bs(\tens{S})$ when $d(\tens{S})\le4$ and,
in arbitrary dimension, when $\bs(\tens{S})\le5$
\cite{seigal2020ranks}.
Ma\'ndziuk and Ventura establish this equality for concise
symmetric cubic tensors of minimal border rank in at most five
variables, and for cubic forms of symmetric border rank at most six
\cite[Proposition~5.14 and Corollary~5.15]{mandziuk2026symmetrization}.
Condition~(ii) of \cref{thm:cubic-partners} covers all symmetric
cubic tensors in at most five essential variables, without a
separate border-rank assumption.
It also establishes both ordinary and symmetric border-rank
additivity with a Hankel summand of arbitrary dimension.
The proof supplies the matrix lower bounds needed for these
additivity statements.
\end{remark}

\section{A criterion for CP rank additivity}\label{sec:criterion}

\subsection{Notation, supports, and matroid union}
We use bold italic uppercase letters for tensors and bold italic
lowercase letters for vectors. Vectors are written as columns.
Bold upright uppercase letters denote matrices and linear operators
in specified bases. Calligraphic uppercase letters denote spaces,
algebras, and finite index sets.
We use $\tens{T}$ for a general tensor, $\tens{S}$ when symmetry is
assumed from the outset, and $\tens{H}$ for a Hankel tensor, with
subscripts to distinguish tensors of the same type.
Scalars, coordinates, polynomials, linear functionals, and abstract
algebra elements use ordinary italic letters.
We write $\Span$ for linear span and $\rank$ for matrix or linear
map rank; the operator $\Stab$ denotes the stabilizer of a subspace.
All transposes are ordinary transposes over $\CC$.

The support in mode $i$ is the image of the flattening obtained by
contracting all other modes. A tensor is concise if each support
equals its ambient factor space. Replacing the factor spaces by these
supports preserves CP rank. This support property holds over
any field; see \cite[Lemma~2.7]{buczynski2020additivity}.
For symmetric tensors, the common support can be used in all modes,
and inclusion and projection also preserve symmetric rank.
The same arguments preserve both border ranks, since linear maps
are continuous and the inclusions have linear retractions.
A tensor in $\bigotimes_i\set{A}_i$ is viewed as a multilinear form
on $\prod_i\set{A}_i^*$.
For $\set{C}\subseteq\set{A}_i$, quotienting the $i$th factor
restricts that input to the annihilator $\set{C}^\perp\subseteq\set{A}_i^*$.
Indeed, the canonical identification is
$\set{A}_i/\set{C}\simeq(\set{C}^\perp)^*$, and
$\codim\set{C}^\perp=\dim\set{C}$.

The following identity provides the common transfer from ordinary
to symmetric rank, both for exact rank and for border rank.
In either identity, vanishing of the three gaps on the right
forces both gaps on the left to vanish. We use this observation
for all symmetric consequences below. 
We omit its simple proof.

\begin{lemma}[Gap transfer]\label{lem:gap-transfer}
For symmetric tensors $\tens{S}_1$ and $\tens{S}_2$ of the same order,
the following identities hold:
\begin{align}
\Delta_c(\tens{S}_1\oplus\tens{S}_2)+\Delta_s(\tens{S}_1,\tens{S}_2)
&=\Delta_c(\tens{S}_1)+\Delta_c(\tens{S}_2)
  +\Delta_r(\tens{S}_1,\tens{S}_2),\label{eq:gap-identity}\\
\underline{\Delta}_c(\tens{S}_1\oplus\tens{S}_2)
+\underline{\Delta}_s(\tens{S}_1,\tens{S}_2)
&=\underline{\Delta}_c(\tens{S}_1)+\underline{\Delta}_c(\tens{S}_2)
  +\underline{\Delta}_r(\tens{S}_1,\tens{S}_2).
\label{eq:border-gap-identity}
\end{align}
\end{lemma}

A vector family indexed by a finite set $\set{E}$ defines a vector
matroid: its independent sets index linearly independent vectors,
and its rank on $\set{S}\subseteq \set{E}$ is the dimension of their span.
We use the standard matroid union theorem
\cite{edmonds1968partition}; the following form is stated directly
in \cite[Theorem~4]{goemans2009union}.

\begin{lemma}[Matroid union]\label{lem:matroid}
Let $M_1,\ldots,M_q$ be matroids on a finite set $\set{E}$, with
rank functions $\varrho_1,\ldots,\varrho_q$.
The largest value of $\sum_i|\set{I}_i|$ over pairwise disjoint sets
$\set{I}_i$ that are independent in $M_i$ is
\begin{equation}\label{eq:matroid-union}
\min_{\set{S}\subseteq \set{E}}
\left(|\set{E}\setminus \set{S}|+\sum_{i=1}^q\varrho_i(\set{S})\right).
\end{equation}
\end{lemma}

\subsection{Eliminating terms while preserving the partner}
\begin{proposition}[Elimination criterion]\label{prop:criterion}
Let $m\ge2$ and
$\tens{T}_1\in\set{A}_1\otimes\cdots\otimes\set{A}_m$.
Write $\tens{T}_{1,i,\set{C}}$ for its image under the quotient
$\set{A}_i\to\set{A}_i/\set{C}$ in mode $i$.
Suppose that, in $m-1$ modes, every subspace $\set{C}$ satisfies
\begin{equation}\label{eq:criterion}
 r(\tens{T}_{1,i,\set{C}})
 \ge r(\tens{T}_1)-(m-1)\dim\set{C}.
\end{equation}
Then $\Delta_r(\tens{T}_1,\tens{T}_2)=0$ for every tensor $\tens{T}_2$
of order $m$.
\end{proposition}
\begin{proof}
The zero case is immediate. Label the controlled modes by
$1,\ldots,m-1$ and write $R=r(\tens{T}_1)$.
Consider a CP decomposition of length $L$ of the form
\[
 \tens{T}_1\oplus\tens{T}_2
 =\sum_{\ell=1}^L\bigotimes_{i=1}^m
               (\vect{a}_{\ell i}+\vect{b}_{\ell i}),
 \ 
 \vect{a}_{\ell i}\in\set{A}_i,\ 
 \vect{b}_{\ell i}\in\set{B}_i.
\]
Let $\set{E}=\{1,\ldots,L\}$, and in mode $i<m$ use the vector
matroid of the original components $\vect{a}_{\ell i}$.
Its rank function is
$\varrho_i(\set{S})=\dim\Span\{\vect{a}_{\ell i}:\ell\in \set{S}\}.$

\emph{Step 1: selecting independent terms.}
For any $\set{S}\subseteq \set{E}$, choose a mode $j<m$ where
$\varrho_j(\set{S})$ is minimal. Project the decomposition onto
all original factors, then quotient mode $j$ by the span
of $\vect{a}_{\ell j}$ with $\ell\in \set{S}$.
Every term in $\set{S}$ vanishes, so \eqref{eq:criterion} implies
\[
 |\set{E}\setminus \set{S}|\ge R-(m-1)\varrho_j(\set{S}).
\]
Since $(m-1)\varrho_j(\set{S})\le\sum_{i=1}^{m-1}\varrho_i(\set{S})$,
every entry of the minimum in \eqref{eq:matroid-union} is
at least $R$. Thus \cref{lem:matroid} supplies pairwise
disjoint independent sets $\set{I}_1,\ldots,\set{I}_{m-1}$ containing
at least $R$ indices in total. Delete indices if necessary
so that their total size is exactly $R$.

\emph{Step 2: eliminating terms.}
For each $i<m$, the vectors
$\{\vect{a}_{\ell i}:\ell\in \set{I}_i\}$ are independent.
Hence there is a linear map
$\matr{F}_i:\set{A}_i\to\set{B}_i$ with
$\matr{F}_i(\vect{a}_{\ell i})=-\vect{b}_{\ell i}
(\ell\in \set{I}_i).$ 
Define it first on the indicated span and extend along a
complement. Set
\[
 \matr{P}_i(\vect{a}+\vect{b})
       =\matr{F}_i(\vect{a})+\vect{b}\ (i<m),
 \ 
 \matr{P}_m(\vect{a}+\vect{b})=\vect{b}.
\]
Each map fixes its partner space pointwise. The last map
annihilates the original space.

The product map $\bigotimes_{i=1}^m\matr{P}_i$ sends
$\tens{T}_1\oplus\tens{T}_2$ to the same tensor $\tens{T}_2$.
For every selected index, its simple term vanishes in its
assigned mode. At least $R$ terms are therefore deleted,
giving $r(\tens{T}_2)\le L-R$.
Taking a minimum joint CP decomposition and using subadditivity
proves the result.
\end{proof}

The proof is inspired by the projection arguments in the proof of
\cite[Lemma~3.5]{buczynski2020additivity} and combines them with
matroid union to eliminate terms in several modes while preserving
the partner.
The same argument gives the following quantitative form.

\begin{corollary}[A criterion with an error term]\label{cor:error}
If the hypotheses of \cref{prop:criterion} are replaced by
\[
r(\tens{T}_{1,i,\set{C}})
 \ge r(\tens{T}_1)-(m-1)\dim\set{C}-\varepsilon
\] 
for a nonnegative integer $\varepsilon$, then we have
$\Delta_r(\tens{T}_1,\tens{T}_2)\le\varepsilon$
for every tensor of the same order.
\end{corollary}
\begin{proof}
The minimum in \eqref{eq:matroid-union} is at least
$r(\tens{T}_1)-\varepsilon$. If this is positive, eliminate
that many indices by the same construction.
Otherwise use
$r(\tens{T}_1\oplus\tens{T}_2)\ge r(\tens{T}_2)$.
\end{proof}

\section{Hankel tensors and their moment algebras}\label{sec:moment}

\subsection{Binary apolar data}
Fix $m\ge3$. Let $\tens{H}\ne\tens{0}$ belong to
$\set{H}_m(n)$, initially with $n\ge2$, and put $D=m(n-1)$.
Identify each input with a polynomial in $\CC[t]_{<n}$,
where the subscript means degree strictly less than $n$.
Define $L(t^j)=h_j$ for $0\le j\le D$. Then
$\tens{H}(p_1,\ldots,p_m)=L(p_1\cdots p_m).$ 
The associated binary form is
$$F_H(x,y)=\sum_{j=0}^D\binom Dj h_jx^{D-j}y^j.$$ 
For $0\le d\le D$, define its degree $d$ apolar space by
\begin{equation*}
 \set{I}_d=
 \left\{\sum_{j=0}^d g_jX^{d-j}Y^j:
       \sum_{j=0}^d g_jh_{i+j}=0\ (0\le i\le D-d)\right\}.
\end{equation*}
The coefficient vectors of $\set{I}_d$ form the kernel of the
degree-$d$ catalecticant. Binary apolarity identifies
$\rho=\rho(\tens{H})$ from \eqref{eq:rho-middle} with the smallest
degree for which $\set{I}_\rho\ne0$; see \cite[Section~2]{nie2019hankel}.
Moreover, the symmetric border rank of $F_H$ over $\CC$
is $\rho$ by \cite[Theorem~2.2]{nie2019hankel}.
In particular, we have $1\le\rho\le\lfloor D/2\rfloor+1$.

Choose a nonzero $G_H\in\set{I}_\rho$.
A common binary coordinate change induces an invertible map
on forms of degree $n-1$ in every tensor mode, preserving all
four tensor ranks. It can be chosen so that the roots of $G_H$ avoid
infinity. After this change, rescale $G_H$ so that
$g_H(t)=G_H(1,t)$ is monic, and write
\begin{equation*}
 g_H(t)=G_H(1,t)=\prod_{\alpha=1}^{\sigma}
                    (t-\zeta_\alpha)^{\mu_\alpha},
 \  \sum_{\alpha=1}^{\sigma}\mu_\alpha=\rho,
\end{equation*}
with $g_H$ monic and the $\zeta_\alpha$ distinct.
Thus $\sigma$ counts the distinct projective roots of $G_H$.
The symbols $h_j$ refer to the transformed moments.
Equivariance also follows by viewing inputs as homogeneous
binary polynomials of degree $n-1$: multiplication of $m$
inputs commutes with the common coordinate change.

\subsection{A moment model and support}
A functional $\lambda:\set{A}\to\CC$ on a finite-dimensional
commutative unital algebra is \emph{Frobenius} if
$(x,y)\mapsto\lambda(xy)$ is nondegenerate.
Products of subspaces mean the linear span of their products.

\begin{lemma}[Moment model and support]\label{lem:moment}
There is a Frobenius functional on $\set{A}=\CC[t]/(g_H)$
such that $\lambda([t^j])=h_j$ for $0\le j\le D$.
The actual support dimension of $\tens{H}$ is
$b=\min(n,\rho)$. On its support, the tensor has the form
\begin{equation*}
 \tens{H}(p_1,\ldots,p_m)=\lambda(p_1\cdots p_m),
 \  p_i\in\set{V}=\CC[t]_{<b}\subseteq\set{A}.
\end{equation*}
With $a=\rho-b$, one has
$a\ge0, 2a\le(m-2)(b-1)$ and $\set{V}^{m-1}=\set{A}.$ 
If $a>0$, then we have $b=n$. If $m\ge4$, we also have
$\set{V}^{m-2}=\set{A}$.
\end{lemma}

\begin{proof}
Write $g_H(t)=\sum_{j=0}^{\rho}g_jt^j$ with $g_\rho=1$ and
define $\lambda([t^j])=h_j$ for $j<\rho$.
The apolar equations $\sum_{j=0}^{\rho}g_jh_{i+j}=0$ for
$0\le i\le D-\rho$ extend these identities to $j\le D$.
The middle catalecticant is the pullback of the pairing
$(u,v)\mapsto\lambda(uv)$ through two surjections onto $\set{A}$,
since both degree bounds are at least $\rho-1$.
Its rank is $\rho=\dim\set{A}$, so the pairing is nondegenerate.
This gives the moment model of \cite[Lemma~2.1]{holden2026exact}
directly from the binary apolar data.

Products of $m-1$ original inputs span $\set{A}$, since
$(m-1)(n-1)\ge\rho-1$.
Therefore the Frobenius pairing identifies the contraction kernel
with the kernel of $\CC[t]_{<n}\to\set{A}$.
The support dimension is $b=\min(n,\rho)$.
Passing to the image preserves all four ranks by the support
invariance established in \cref{sec:criterion}. 
If $a=0$, then we have $\set{V}=\set{A}$ and the remaining claims follow.
Otherwise we have $b=n$, and the apolar degree bound gives
$2a\le(m-2)(b-1)$.
For $m\ge4$, we also have
\[
\rho-1\le\lfloor m(b-1)/2\rfloor\le(m-2)(b-1).
\] 
Thus products of $m-2$ inputs already span $\set{A}$.
\end{proof}

To include the scalar case $n=1$, set
$(b,\rho,\sigma,a)=(1,1,1,0)$ and $\set{A}=\set{V}=\CC$,
with $\lambda(z)=h_0z$. 

\subsection{The border upper bound and the coefficient pairing}
The linear Hankel embedding sends the binary power
$(\alpha x+\beta y)^D$ to the symmetric tensor
$(\alpha^{n-1},\alpha^{n-2}\beta,\ldots,\beta^{n-1})^{\otimes m}.$ 
Applying it to binary border decompositions gives
\begin{equation}\label{eq:border-upper}
\br(\tens{H})\le\bs(\tens{H})\le\rho.
\end{equation}
This is the upper-bound argument of
\cite[Theorem~2.2, Lemma~3.2, and Section~4]{nie2019hankel};
it uses complex binary apolarity and does not require an exact
Hankel rank formula. 

We recall the standard coefficient pairing on a monogenic quotient algebra and record its orthogonality property for the degree filtration. We include the short proof for completeness.

\begin{lemma}[Coefficient pairing]\label{lem:coefficient-pairing}
Let $g$ be monic of degree $M\ge1$ and set
$\set{B}=\CC[t]/(g)$, represented by polynomials of degree less
than $M$. The functional $\lambda_0(q)=[t^{M-1}]q$ is Frobenius.
Here $\perp$ denotes orthogonality in $\set{B}$ for the pairing
$(u,v)\mapsto\lambda_0(uv)$. For $0\le a_0<M$, we have
\begin{equation}\label{eq:window-orthogonal}
(\CC[t]_{<M-a_0})^\perp=\CC[t]_{<a_0}.
\end{equation}
\end{lemma}
\begin{proof}
A nonzero remainder of degree $j$ pairs nontrivially with
$t^{M-1-j}$, proving nondegeneracy. Products of elements in
the two displayed spaces have degree at most $M-2$, proving
orthogonality. Their dimensions are complementary.
\end{proof}
\subsection{The Hankel rank formula}
We use the following formula from
\cite[Theorem~1.1]{holden2026exact}, written in the moment coordinates
of \cref{lem:moment}.

\begin{proposition}[Hankel rank formula]\label{prop:upper}
Retain the notation of Lemma~\ref{lem:moment}, with
$b=\min(n,\rho)$ and $a=\rho-b$. Set
\begin{equation}\label{eq:U}
U\eqdef\min\bigl\{(m-1)b-a-(m-2),\
(m-1)\rho-(m-2)\sigma\bigr\}.
\end{equation}
Then $r(\tens{H})=s(\tens{H})=U$.
Moreover, the moment parameters satisfy
\begin{equation}\label{eq:first-threshold}
(m-1)b-a-(m-2)\ge\rho\ge b.
\end{equation}
\end{proposition}

To compare the formulas, note that $b=n$ when $\rho>n$.
The first entry is then $m(n-1)-\rho+2$.
If $\rho\le n$, then we have $b=\rho$ and $a=0$, and the second
entry is no larger than either version of the first.
Finally, subtracting $\rho$ from the first entry gives
$$(m-2)(b-1)-2a\ge0,$$ 
proving \eqref{eq:first-threshold}.
The value is independent of the choice of a minimum apolar polynomial
by \cite[Theorem~1.1]{holden2026exact}.

\section{A partition lower bound}\label{sec:partition}

\subsection{Product growth with a Frobenius constraint}
For a subspace $\set{W}$ of a commutative algebra $\set{B}$, write
$\Stab_{\set{B}}(\set{W})\eqdef
\{z\in\set{B}:z\set{W}\subseteq\set{W}\}$.
Recall that an algebra is reduced if its only nilpotent element is zero.
The following is a consequence of
\cite[Lemma~4.2 and Corollary~4.3]{holden2026exact}, whose proof uses the
transformation argument of \cite[Lemma~4.2]{beck2017additive}.

\begin{lemma}[Constrained product growth]\label{lem:growth}
Let $m\ge2$ and let $\set{B}=\set{A}\times\CC^R$, where
$\set{A}$ is a finite-dimensional commutative complex algebra.
Let $\set{U}_1,\ldots,\set{U}_m\subseteq\set{B}$ contain $1$.
Suppose every product omitting one input space projects onto $\set{A}$.
Assume that a functional $\Lambda$ annihilates
$\set{W}=\prod_{i=1}^m\set{U}_i$ and that its restriction to
$\set{A}$ is Frobenius.
Then $\Stab_{\set{B}}(\set{W})$ is reduced.
If this stabilizer equals $\CC1$, then
\begin{equation*}
\dim\set{W}\ge\sum_{i=1}^m\dim\set{U}_i-(m-1).
\end{equation*}
\end{lemma}

\subsection{The partition associated with a CP decomposition}
Let $\set{A}=\prod_{\alpha\in \set{I}}\set{A}_\alpha$ be the
decomposition of a commutative Frobenius algebra into
local factors, with functional $\lambda$.
Suppose $\set{V}_1,\ldots,\set{V}_m\subseteq\set{A}$
contain units and
\begin{equation}\label{eq:omitted-products}
 \prod_{i\ne j}\set{V}_i=\set{A}\ (1\le j\le m).
\end{equation}
For a nonempty $\set{J}\subseteq \set{I}$, write
$\set{A}_{\set{J}}=\prod_{\alpha\in \set{J}}\set{A}_\alpha$ and let
$\pi_{\set{J}}$ be the projection onto it. Put
$M_{\set{J}}=\dim\set{A}_{\set{J}}$ and $d_i(\set{J})=\dim\pi_{\set{J}}\set{V}_i.$ 

\begin{proposition}[Partition lower bound]\label{prop:partition}
The multilinear form
$\tens{T}(v_1,\ldots,v_m)=\lambda(v_1\cdots v_m)$
on $\prod_{i=1}^m\set{V}_i$ satisfies
\begin{equation}\label{eq:partition}
 r(\tens{T})\ge
 \min_{\set{P}}\sum_{\set{J}\in\set{P}}
 \left(\sum_{i=1}^m d_i(\set{J})-M_{\set{J}}-(m-2)\right),
\end{equation}
where the minimum is over partitions of the index set $\set{I}$
of local factors.
\end{proposition}
\begin{proof}
\emph{Step 1: a graph representation.}
Choose a minimum CP decomposition of length $R$, written as
$\tens{T}=\sum_{\ell=1}^R\bigotimes_{i=1}^m\beta_{\ell i}$.
All factor functionals are nonzero.
In $\set{B}=\set{A}\times\CC^R$ take the graph spaces
\[
 \widetilde{\set{U}}_i
 =\{(v,\beta_{1i}(v),\ldots,\beta_{Ri}(v)):
                                      v\in\set{V}_i\}.
\]
Their product is annihilated by
$\Lambda_0(x,\vect{z})=\lambda(x)-\sum_{\ell=1}^Rz_\ell$.
Each graph contains a unit $\gamma_i$: vanishing at any residue
of $\set{A}$ or any coordinate of $\CC^R$ defines a proper linear
subspace of $\set{V}_i$, and finitely many such subspaces cannot
cover it. We call these the moment and decomposition coordinates.
Set
$\set{U}_i=\gamma_i^{-1}\widetilde{\set{U}}_i$ and 
$\Lambda(x)=\Lambda_0(\gamma_1\cdots\gamma_m x).$ 
Then $1\in\set{U}_i$ and $\Lambda$ annihilates
$\set{W}=\prod_i\set{U}_i$.
The restriction of $\Lambda$ to $\set{A}$ remains
Frobenius and its coefficient on every decomposition
coordinate is nonzero.
Multiplication by units preserves each projected
dimension $d_i(\set{J})$ and the hypotheses on omitted products.

\emph{Step 2: surjective projections.}
The moment projection of $\set{W}$ is surjective by
\eqref{eq:omitted-products}.
The projection onto $\CC^R$ is also surjective.
Indeed, the simple tensors in a minimum CP decomposition
are linearly independent: a dependence would remove one
term after rescaling the others.
Dually, their simultaneous evaluation map has image
$\CC^R$. The normalization changes its coordinates
only by nonzero scalar factors.
By \cref{lem:growth}, the stabilizer
$\set{D}=\Stab_{\set{B}}(\set{W})$ is reduced.

\emph{Step 3: the stabilizer blocks.}
For a primitive idempotent $e$ of $\set{D}$ put
$\set{B}_e=e\set{B}$ and $\set{W}_e=e\set{W}$.
Every block contains both moment factors and decomposition
coordinates. Otherwise the surjectivity onto the surviving
side would imply $\set{W}_e=\set{B}_e$, contradicting
$\Lambda(\set{W}_e)=0$: on a block consisting of decomposition
coordinates, the functional has nonzero coefficients; on a block
consisting of moment factors, its restriction is Frobenius. 
The moment parts therefore form a partition $\set{P}$
of $\set{I}$. Its block for $\set{J}$ is
$\set{A}_{\set{J}}\times\CC^{R_{\set{J}}}$, with $\sum_{\set{J}}R_{\set{J}}=R$.
Moreover, we have
\[
 \set{W}_e=\prod_{i=1}^m(e\set{U}_i),\ 
 \Stab_{\set{B}_e}(\set{W}_e)=e\set{D}=\CC e.
\]
For the stabilizer equality, extend any block
stabilizer by zero: it stabilizes $\set{W}$ and hence
belongs to $e\set{D}$. The other inclusion is immediate.

\emph{Step 4: the cost of a block.}
The spaces $e\set{U}_i$ contain the unit $e$ of
$\set{B}_e$, and every omitted product projects onto
$\set{A}_{\set{J}}$. The restricted functional is Frobenius
on the moment part.
Therefore \cref{lem:growth} applies to these projected spaces.
Since $\dim(e\set{U}_i)\ge d_i(\set{J})$, we obtain
\[
 \sum_{i=1}^m d_i(\set{J})-(m-1)
 \le\dim\set{W}_e\le M_{\set{J}}+R_{\set{J}}-1.
\]
The upper inequality uses the nonzero annihilating
functional on $\set{B}_e$.
Rearrange and sum over the actual partition to obtain
a lower bound for $R$. Taking the minimum over all
partitions proves \eqref{eq:partition}.
\end{proof}

The graph and stabilizer construction adapts the partition
framework of \cite[Section~5]{holden2026exact}.
When $\set{V}_i=\set{A}$ for every $i$, \eqref{eq:partition}
reduces to
$r(\tens{T})\ge(m-1)\dim\set{A}-(m-2)|\set{I}|$,
with the minimum attained at the partition into individual
local factors.
Under the Frobenius identification $\set{A}\simeq\set{A}^*$,
this is the multilinear Alder--Strassen lower bound for
$(m-1)$-fold multiplication
\cite[Theorem~19]{blaser2023multilinear}.
Keeping the separate dimensions $d_i(\set{J})$ allows
restricted and unequal input spaces, as needed for the
uniform restriction bound proved below.

\section{Polynomial subspaces and input restrictions}
\label{sec:restriction}

\subsection{A dimension estimate for block projections}
Use the coefficient pairing from \cref{lem:coefficient-pairing}
on $\set{B}=\CC[t]/(g)$, where $g$ is monic of degree $M$.

\begin{lemma}[Block count for a polynomial subspace]\label{lem:window}
Let $0\le a_0<M$, let $\set{S}\subseteq\CC[t]_{<M-a_0}$ have dimension $k$, and
put $z=M-a_0-k$. Partition the local factors of $\set{B}$ into
$\ell$ nonempty parts. If
\[
 \set{L}=\bigoplus_{\set{J}}\pi_{\set{J}}\set{S},\ 
 \eta=\dim\set{L}-k,
\]
then
$\eta\ge a_0$ or 
$\ell\le a_0+z.$ 
\end{lemma}
\begin{proof}
Suppose $\eta<a_0$. By \eqref{eq:window-orthogonal}, pairing
with $\CC[t]_{<a_0}$ vanishes on $\set{S}$, so the map
$\CC[t]_{<a_0}\to\set{L}^*$ has rank at most $\eta$.
Its kernel $\set{N}$ therefore has dimension
$\tau\ge a_0-\eta>0$.
Let $q=\gcd(g,\set{N})$, the monic greatest common divisor
of $g$ and a basis of $\set{N}$. Since all elements of
$\set{N}$ are divisible by $q$ and have degree less than $a_0$,
we have
$\deg q\le a_0-\tau.$ 
If $\pi_{\set{J}}\set{N}=0$, the polynomial defining that whole part
divides $q$. These part polynomials are pairwise coprime and
nonconstant. Consequently at most $a_0-\tau$ parts have zero
projection, and
$\sum_{\set{J}}\dim\pi_{\set{J}}\set{N}\ge\ell-a_0+\tau.$ 
Different algebra factors are orthogonal for $\lambda_0$.
Since $\set{L}$ is the direct sum of its projections,
the orthogonality $\set{N}\perp\set{L}$ implies
$\bigoplus_{\set{J}}\pi_{\set{J}}\set{N}\subseteq\set{L}^{\perp}$.
Nondegeneracy now gives
\[
 a_0+z-\eta=M-\dim\set{L}
 \ge\sum_{\set{J}}\dim\pi_{\set{J}}\set{N}
 \ge\ell-a_0+\tau\ge\ell-\eta.
\]
Cancellation proves $\ell\le a_0+z$.
\end{proof}

\subsection{The restriction bound}
Return to the moment data in \cref{lem:moment} and the
rank value $U$ in \eqref{eq:U}.

\begin{proposition}[Uniform CP rank restriction bound]\label{prop:restriction}
Let $m\ge3$.
For every subspace $\set{K}\subseteq\set{V}$ of codimension $u$
in any one input mode, we have
\begin{equation}\label{eq:restriction}
 r\bigl(\tens{H}|_{\set{K}\times\set{V}\times\cdots\times\set{V}}\bigr)
 \ge U-(m-1)u.
\end{equation}
By \cref{prop:upper}, the right side equals
$r(\tens{H})-(m-1)u$.
\end{proposition}
\begin{proof}
\emph{Step 1: zero and small inputs.}
Put $k=b-u$.
If $k=0$, the bound is nonpositive since
$U-(m-1)b\le-a-(m-2)$.
When $m=3$ and $k\le a+1$, the first contraction
of the restricted tensor is injective, by conciseness
on $\set{V}$, so its flattening rank is $k$.
On the other hand
$U-2u\le2k-a-1\le k$, proving the assertion.
In what follows $k>0$, and for $m=3$ we also have
$k>a+1$.

\emph{Step 2: removing the common divisor.}
Let
$f=\gcd(g_H,\set{K}), d_f=\deg f, 
 g_f=g_H/f$ and  $M=\rho-d_f.$ 
The greatest common divisor can be computed from any
basis of $\set{K}$. Since every element is divisible
by $f$, we have $k\le b-d_f$ and $d_f\le u$.
Division by $f$ identifies $\set{K}$ with a
$k$-dimensional space
\[
 \set{S}\subseteq\CC[t]_{<b-d_f}
       =\CC[t]_{<M-a}\subseteq\set{B}=\CC[t]/(g_f).
\]
The last inclusion is injective because $b-d_f\le M$.
Set $z=b-d_f-k=u-d_f$.
Since $k>0$, the window parameter satisfies $a<M$. 
The space $\set{S}$ contains a unit modulo $g_f$.
At each root of $g_f$, the multiplicity in $f$ is the
minimum of the multiplicities in $g_H$ and in the
polynomials of $\set{K}$.
That minimum is strictly less than the multiplicity
in $g_H$, so some divided element is nonzero at the
root. Avoiding these finitely many vanishing hyperplanes gives
a unit in $\set{S}$.

Define $\lambda_f(q)=\lambda(fq)$ on $\set{B}$.
This is well defined and Frobenius even if $f$ and
$g_f$ have common roots. Indeed, the annihilator of
$f$ in $\CC[t]/(g_H)$ is the ideal generated by $g_f$.
If $\lambda(fqv)=0$ for all $v$ modulo $g_f$,
it also vanishes for all $v$ modulo $g_H$.
Nondegeneracy gives $fq=0$ modulo $g_H$, hence
$q=0$ modulo $g_f$. 
Let $\set{V}_0$ be the image of $\CC[t]_{<b}$ in
$\set{B}$.
The original restriction is the pullback of
$\lambda_f(sv_2\cdots v_m)$ on
$\set{S}\times\set{V}_0\times\cdots\times\set{V}_0$.
The last $m-1$ input maps are surjective and have
linear sections. Their pullback therefore preserves
CP rank exactly.

\emph{Step 3: verifying all omitted products.}
For $m\ge4$, \cref{lem:moment} gives
$\set{V}_0^{m-2}=\set{B}$ after passing to the
quotient.
Since $1\in\set{V}_0$ and $\set{S}$ contains a unit,
every product omitting one of the $m$ inputs spans
$\set{B}$. 
For $m=3$ one has $\set{V}_0^2=\set{B}$.
If $M\le b$, then we have $\set{V}_0=\set{B}$ and the
unit in $\set{S}$ gives $\set{S}\set{V}_0=\set{B}$.
If $M>b$, consider
\[
 \set{S},\  \set{S}+t\set{S},\ \ldots,\ 
 \set{S}+t\set{S}+\cdots+t^j\set{S}.
\]
A stalled space is $t$-invariant and contains a
unit. It therefore contains that unit times every
element of $\CC[t]/(g_f)$, so equals $\set{B}$.
The dimensions increase by at least one until that
happens. Consequently
\[
 \dim(\set{S}\set{V}_0)\ge\min\{M,k+b-1\}=M,
\]
using $k>a+1$ and $M=b+a-d_f$.
Thus \cref{prop:partition} applies in every case
still under consideration.

\emph{Step 4: partitions containing a long part.}
Fix any partition of the whole local factors of
$\set{B}$.
For a part $\set{J}$, let $t_{\set{J}}=\dim\set{B}_{\set{J}}$ and
$d_S(\set{J})=\dim\pi_{\set{J}}\set{S}$.
Polynomial division gives
$\dim\pi_{\set{J}}\set{V}_0=\min\{b,t_{\set{J}}\}$.
Its cost in \eqref{eq:partition} is
\begin{equation}\label{eq:part-cost}
 d_S(\set{J})+(m-1)\min\{b,t_{\set{J}}\}-t_{\set{J}}-(m-2).
\end{equation} 
If $t_{\set{J}}\ge b$, then we have $d_S(\set{J})\ge k-(M-t_{\set{J}})$,
so this part alone has cost at least
\begin{align*}
 k+(m-1)b-M-(m-2)
 =(m-1)b-a-(m-2)-u+d_f
 \ge U-(m-1)u.
\end{align*}
All other parts have nonnegative cost.
A short part has cost
$d_S(\set{J})+(m-2)(t_{\set{J}}-1)\ge1$, because
$\set{S}$ contains a unit.
For another long part, use $d_S(\set{J})\ge1$
and $t_{\set{J}}\le\rho$ to bound its cost below by
\[
 1+(m-1)b-\rho-(m-2)
 =1+\bigl((m-1)b-a-(m-2)\bigr)-b\ge1,
\]
where the last inequality is \eqref{eq:first-threshold}.

\emph{Step 5: partitions with only short parts.}
Suppose all $t_{\set{J}}<b$.
Let $\ell$ be the number of parts, let $\sigma_f$
be the number of distinct roots of $g_f$, and set
$\eta=\sum_{\set{J}}d_S(\set{J})-k\ge0$.
Summing \eqref{eq:part-cost} gives
\[
 k+(m-2)M+\eta-(m-2)\ell.
\]
If $\eta\ge a$, use $\ell\le\sigma_f\le\sigma$
to obtain
\begin{align*}
 &k+(m-2)M+\eta-(m-2)\ell
 \ge k+(m-2)M+a-(m-2)\sigma_f\\
 &=(m-1)\rho-(m-2)\sigma-(m-1)u +(m-2)(u-d_f)+(m-2)(\sigma-\sigma_f)
 \ge U-(m-1)u.
\end{align*}
If $\eta<a$, \cref{lem:window} gives
$\ell\le a+z$. The cost is then at least
\[
 k+(m-2)(M-a-z)+\eta=(m-1)k+\eta
 \ge U-(m-1)u.
\]
Here $M-a-z=k$ and
$U\le(m-1)b-a-(m-2)\le(m-1)b$.

The cost of every partition has been bounded.
Taking the minimum in \cref{prop:partition} proves
\eqref{eq:restriction}.
Symmetry gives the same bound in every mode.
\end{proof}

\subsection{Proofs of the exact-rank results}

We can now complete the proofs of
\cref{thm:main,cor:symmetric,cor:pair}.

\begin{proof}[Proof of \cref{thm:main}]
A zero summand is immediate.
Otherwise pass to actual supports of the Hankel tensor.
In the moment presentation its tensor factor is
$\set{V}^*$. For any $\set{C}\subseteq\set{V}^*$,
quotienting that factor restricts its input to
$\set{K}=\set{C}^{\perp}\subseteq\set{V}$,
of codimension $\dim\set{C}$.
By \cref{prop:restriction}, we have
\[
 r(\tens{H}_{i,\set{C}})
 \ge r(\tens{H})-(m-1)\dim\set{C}
\]
in every mode. Thus \cref{prop:criterion} applies
in any $m-1$ modes, proving CP rank additivity.
Passage to supports preserves the ranks of the
individual tensors and their direct sum.
\end{proof}

\begin{proof}[Proofs of \cref{cor:symmetric,cor:pair}]
By \cref{prop:upper}, we have $\Delta_c(\tens{H})=0$, while
\cref{thm:main} gives $\Delta_r(\tens{H},\tens{S})=0$.
Thus the right side of \eqref{eq:gap-identity} is $\Delta_c(\tens{S})$.
If $\Delta_c(\tens{S})=0$, both gaps on the left vanish by
nonnegativity, proving \cref{cor:symmetric}.
For a Hankel partner, \cref{prop:upper} supplies this rank
equality and yields \cref{cor:pair}.
\end{proof}

\section{Border rank of two Hankel summands}\label{sec:border}

This section combines the binary border upper bound
\eqref{eq:border-upper} with compatible matrix lower bounds. 
The single-tensor certificates are known: the even-order one is
in \cite[Lemma~4.1 and Theorem~4.2]{nie2019hankel}, and the fixed
odd-order identity below is in the proof of
\cite[Theorem~6.1]{holden2026exact}. 
% We give the latter in our
% notation and then combine the certificates for direct sums.

\subsection{Hankel matrix certificates}
For a third-order tensor $\tens{T}\in\CC^3\otimes\CC^u\otimes\CC^v$
with slices $\matr{M}_0,\matr{M}_1,\matr{M}_2$, define
\begin{equation}\label{eq:koszul-matrix}
\matr{K}(\tens{T})\eqdef
\begin{pmatrix}
0&\matr{M}_2&-\matr{M}_1\\
-\matr{M}_2&0&\matr{M}_0\\
\matr{M}_1&-\matr{M}_0&0
\end{pmatrix}.
\end{equation}
The standard Koszul flattening bound gives
\begin{equation}\label{eq:koszul-border-bound}
\rank\matr{K}(\tens{T})\le2\br(\tens{T})
\end{equation}
by \cite[Theorem~2.1]{landsberg2015koszul}, with first-factor
dimension three and exterior degree one. 
% Grouping modes and
% applying linear maps to the grouped factors cannot increase border
% rank: each original simple tensor maps to a simple tensor or zero,
% and the maps are continuous.

\begin{proposition}[Hankel matrix certificates]\label{prop:border-certificate}
Let $\tens{H}\in\set{H}_m(n)$ with $m\ge3$.
If $m=2q$, its $q\mid q$ flattening has rank $\rho(\tens{H})$.
If $m=2q+1$, after grouping the modes as $1,q,q$ there are
linear maps on the three factors with image $\tens{H}'$ in a
tensor space whose first factor is $\CC^3$, such that
\begin{equation}\label{eq:koszul-exact}
\rank\matr{K}(\tens{H}')=2\rho(\tens{H}).
\end{equation}
\end{proposition}
\begin{proof}
For $m=2q$, selecting one multi-index of each possible index sum
in either group gives the middle binary catalecticant. All other
rows and columns repeat these, so the ranks agree. 
For $m=2q+1$, define the local parameters by
\[
\ell=n-1,\ \  N=q\ell+1,\ \ 
c=\lfloor\ell/2\rfloor,\ \  \delta_+=\ell-c.
\] 
In either group of $q$ modes select one coordinate of each index
sum $0,\ldots,N-1$; in the singleton factor apply
$\vect{z}\mapsto(z_0,z_c,z_\ell)^{\mathsf T}$.
Set $s_0=0$, let $s_1=c$, and let $s_2=\ell$.
The resulting tensor $\tens{H}'$ has slices
$\matr{M}_\nu=(h_{i+k+s_\nu})_{0\le i,k<N}$ for $\nu=0,1,2$,
so 
$$\matr{K}(\tens{H}')=\matr{K}(\matr{M}_0,\matr{M}_1,\matr{M}_2).$$
Repeated coordinate selections for $n=1,2$ are valid linear maps.

Write $\matr{C}=(h_{i+j})_{0\le i<N+c,\,0\le j<N+\delta_+}$
for the middle binary catalecticant. Define
\begin{equation}\label{eq:fixed-map}
\matr{E}:\CC^{N+c}\oplus\CC^{N+\delta_+}\to\CC^{3N},\ \ (\vect{x},\vect{y})\mapsto\bigl((x_{i+c})_{i=0}^{N-1},\,
(-x_i+y_{i+\delta_+})_{i=0}^{N-1},\,(-y_i)_{i=0}^{N-1}\bigr)^{\mathsf T}.
\end{equation}
Block multiplication, using $c+\delta_+=\ell$, gives
\begin{equation}\label{eq:fixed-identity}
\matr{K}(\tens{H}')=
\matr{E}\begin{pmatrix}0&\matr{C}\\-\matr{C}^{\mathsf T}&0\end{pmatrix}
\matr{E}^{\mathsf T}.
\end{equation}
Moreover, the map $\matr{E}$ is injective. Indeed, the equality $\matr{E}(\vect{x},\vect{y})=0$
first gives $x_i=0$ for $i\ge c$ and $y_i=0$ for $i<N$.
For $i<c$, the middle equations give $x_i=y_{i+\delta_+}=0$,
because $i+\delta_+<\ell<N$. They then force all remaining $y_i$ to
vanish, since $\delta_+\le N$.
An injective congruence preserves matrix rank, so
\eqref{eq:fixed-identity} has rank $2\rank\matr{C}=2\rho(\tens{H})$.
The identity holds for every moment vector, including zero.
\end{proof}

The single-tensor equality for all complex Hankel tensors is already
\cite[Theorem~6.1]{holden2026exact}. The fixed identity above also
supplies compatible sharp certificates for the following direct-sum
argument, including repeated roots and unequal dimensions.

\subsection{Combining the matrix lower bounds}
The next elementary construction is stated for finitely many
summands so that the later finite-sum consequence requires no
new matrix argument.

\begin{lemma}[A common first factor]\label{lem:joint-koszul}
Let $\tens{T}_1,\ldots,\tens{T}_k$ have the same order and fix
a grouping of their modes into three nonempty groups.
Suppose factorwise linear maps send each grouped tensor to
$\tens{T}'_\nu\in\CC^3\otimes\set{U}_\nu\otimes\set{W}_\nu$.
Then
\begin{equation}\label{eq:joint-koszul}
\br\left(\bigoplus_{\nu=1}^k\tens{T}_\nu\right)
\ge\frac12\sum_{\nu=1}^k\rank\matr{K}(\tens{T}'_\nu).
\end{equation}
\end{lemma}
\begin{proof}
Group the modes of the direct sum. On the first grouped factor,
annihilate mixed-source blocks and sum the prescribed maps on
the pure blocks into one common $\CC^3$.
On the other two factors, again annihilate mixed-source blocks
and retain the images in separate spaces $\set{U}_\nu$ and
$\set{W}_\nu$. These are linear maps on the grouped factors.
If the slices of $\tens{T}'_\nu$ are $\matr{M}_{\nu,j}$ for
$j=0,1,2$, the resulting tensor $\widehat{\tens{T}}$ has slices
$\bigoplus_\nu\matr{M}_{\nu,j}$.
After row and column permutations, the matrix $\matr{K}(\widehat{\tens{T}})$ is
$\bigoplus_\nu\matr{K}(\tens{T}'_\nu)$.
Thus its matrix rank is the sum of the individual matrix ranks,
whereas $\br(\widehat{\tens{T}})$ does not exceed the border rank of the
original direct sum. Then we apply \eqref{eq:koszul-border-bound}.
\end{proof}

The common first factor keeps the divisor in
\eqref{eq:koszul-border-bound} equal to two. In particular, the
argument applies to arbitrary limiting CP decompositions;
they need not respect the summand blocks.

\begin{proof}[Proof of \cref{thm:border}]
The upper bound \eqref{eq:border-upper} and subadditivity give
\begin{equation}\label{eq:joint-upper}
\br(\tens{H}_1\oplus\tens{H}_2)\le\bs(\tens{H}_1\oplus\tens{H}_2)
\le\rho(\tens{H}_1)+\rho(\tens{H}_2).
\end{equation}
For $m=2q$, the $q\mid q$ flattening has the two individual
flattenings as diagonal blocks and has zero mixed blocks.
By \cref{prop:border-certificate}, its rank is
$\rho(\tens{H}_1)+\rho(\tens{H}_2)$, giving the reverse inequality.
For $m=2q+1$, apply \cref{lem:joint-koszul} to the $1,q,q$
compressions supplied by \cref{prop:border-certificate}.
Their Koszul ranks are $2\rho(\tens{H}_1)$ and $2\rho(\tens{H}_2)$,
so the same lower bound follows. Zero summands contribute zero
throughout. This proves both equalities in \eqref{eq:border-finite}.
\end{proof}

\begin{proof}[Proofs of \cref{cor:hankel-border-ranks,cor:border-symmetric}]
Take $\tens{H}_1=\tens{H}$ and $\tens{H}_2=0$ in \cref{thm:border}. Invariance under adjoining
zero coordinates, as noted in \cref{sec:criterion}, gives
\cref{cor:hankel-border-ranks} and hence
$\underline{\Delta}_c(\tens{H})=0$.
Under the hypotheses of \cref{cor:border-symmetric}, the right
side of \eqref{eq:border-gap-identity} therefore vanishes.
Nonnegativity forces both gaps on the left to vanish, and the
single-tensor formula gives the stated common value.
When $\tens{S}$ is Hankel, its hypotheses follow from
\cref{thm:border,cor:hankel-border-ranks}.
\end{proof}

\section{Koszul bounds for third-order partners}\label{sec:cubic-bounds}

We first apply the common-factor construction to partners of small
border rank. We then obtain a five-slice Hankel bound for the
symmetric partners considered in \cref{sec:cubic-symmetric}.
Throughout these two sections, we fix $\tens{H}\in\set{H}_3(n)$
and write $\rho=\rho(\tens{H})$.

\subsection{The common matrix bound and low border rank}
For a third-order tensor $\tens{T}$, let $\tens{T}'$ range over
the tensors obtained by permuting its factors and mapping the
first factor linearly to $\CC^3$. Define its three-slice Koszul
bound by 
\[
\kappa(\tens{T})\eqdef
\max_{\tens{T}'}\lceil\rank\matr{K}(\tens{T}')/2\rceil.
\]
The maximum exists because the possible matrix ranks form a finite
set of integers. By \eqref{eq:koszul-border-bound}, we have $\kappa(\tens{T})\le\br(\tens{T})$.
Moreover, \cref{prop:border-certificate,lem:joint-koszul} give
$$\br(\tens{H}\oplus\tens{T})\ge\rho+\kappa(\tens{T}).$$
Thus a sharp three-slice bound for the partner proves ordinary
border-rank additivity. If the partner is a symmetric tensor $\tens{S}$ and
$\kappa(\tens{S})=\bs(\tens{S})$, the same bound and symmetric
subadditivity prove equality of both border ranks of the direct sum.

\begin{proposition}[Partners of border rank at most three]
\label{prop:small-border-partner}
Every complex third-order tensor $\tens{T}$ with
$\br(\tens{T})\le3$ satisfies $\kappa(\tens{T})=\br(\tens{T})$.
\end{proposition}
\begin{proof}
For two slices $\matr{M}_0,\matr{M}_1$, with the third slice zero,
\eqref{eq:koszul-matrix} gives
\[
\rank\matr{K}(\matr{M}_0,\matr{M}_1,0)
=\rank\binom{\matr{M}_0}{\matr{M}_1}
+\rank(\matr{M}_0\ \matr{M}_1).
\]
Indeed, the first two block columns and the last block column
have images in disjoint block rows. The two ranks on the right
are the support dimensions of the remaining factors.

The zero tensor is immediate. Otherwise pass to supports and
write their dimensions in increasing order as $d_1\le d_2\le d_3$.
Each $d_i\le\br(\tens{T})$. For border rank one all three
dimensions are one, and the two-slice formula gives rank two.
For border rank two, at least two dimensions are two, so the same
formula gives rank four after embedding the first factor in $\CC^3$.
Suppose now that $\br(\tens{T})=3$.
Then $d_3=3$: every complex $2\times2\times2$ tensor has border
rank at most two, since a dense set of matrix pencils has an
invertible slice and two distinct generalized eigenvalues.
If $d_1\le2$, the possible formats are
$\{(1,3,3),(2,2,3),(2,3,3)\}$. The two-slice formula gives ranks six, five,
and six, respectively, whose half-ranks all round up to three. 
It remains to treat the concise $3\times3\times3$ case.
Write $[ijk]=\vect{a}_i\otimes\vect{b}_j\otimes\vect{c}_k$ in
bases of the three supports. By
\cite[Theorem~1.2]{buczynski2014third}, the four normal forms,
up to permutations of the factors, are
\begin{align*}
&[111]+[222]+[333],\ \ 
[211]+[121]+[112]+[333],\\ 
&[122]+[212]+[221]+[113]+[131]+[311],\ \ [212]+[221]+[113]+[131]+[311].
\end{align*}
Substitution in \eqref{eq:koszul-matrix} gives rank six in each
case. Hence $\kappa(\tens{T})=3$, completing the proof.
\end{proof}

\subsection{A five-slice Hankel bound}
We use the higher Koszul map only for the remaining symmetric
partners. For $\tens{T}\in\set{A}\otimes\set{B}\otimes\set{C}$,
write $\matr{K}_{\set{A}}^{p}(\tens{T})$ for the linear map from
$\set{B}^*\otimes\bigwedge\nolimits^p\set{A}$ to
$\set{C}\otimes\bigwedge\nolimits^{p+1}\set{A}$ which sends
$\beta\otimes\vect{\omega}$ to $\beta(\vect{v})\vect{w}\otimes(\vect{u}\wedge\vect{\omega})$
when $\tens{T}=\vect{u}\otimes \vect{v}\otimes \vect{w}$, and extend linearly.
If $\dim\set{A}=q$, its rank on a nonzero simple tensor is
$\binom{q-1}{p}$; consequently \cite[Theorem~2.1]{landsberg2015koszul} 
\[
\rank\matr{K}_{\set{A}}^{p}(\tens{T})
\le\binom{q-1}{p}\br(\tens{T}).
\]
For $q=3$ and $p=1$ this agrees, up to choices of bases, with
\eqref{eq:koszul-matrix}. For $q=5$ and $p=2$ the divisor is six.

\begin{proposition}[Five-slice Hankel bound]\label{prop:five-slice}
There is a linear map on the first factor of $\tens{H}$ with image
$\widehat{\tens{H}}\in\CC^5\otimes\CC^n\otimes\CC^n$ such that
$\rank\matr{K}_{\CC^5}^{2}(\widehat{\tens{H}})\ge6\rho-2$.
Consequently, for every $\tens{T}\in\CC^5\otimes\set{B}\otimes\set{C}$,
one has 
\[
\br(\tens{H}\oplus\tens{T})\ge
\rho+\lceil(\rank\matr{K}_{\CC^5}^{2}(\tens{T})-2)/6\rceil.
\]
\end{proposition}
\begin{proof}
The case $\tens{H}=0$ follows directly from the Koszul bound,
so assume $\rho>0$. Put $\ell=n-1$ and
$\set{R}_j=\CC[x,y]_j$. The moments define
$L:\set{R}_{3\ell}\to\CC$ by $L(x^{3\ell-k}y^k)=h_k$.
The pairing $L(uv)$ on $\set{R}_j\times\set{R}_{3\ell-j}$ has rank
$\min\{\rho,j+1,3\ell-j+1\}$.
Indeed, after the binary coordinate change in \cref{sec:moment},
this is the Frobenius pairing of \cref{lem:moment} pulled back
through the two polynomial spaces. At least one degree is at
least $\rho-1$, so that space maps onto the moment algebra;
the other image has the stated dimension. Binary coordinate
changes preserve the ranks of these pairings.

Choose nonnegative integers $\delta_1,\ldots,\delta_4$ with sum
$\ell$ and pairwise differences at most one. Set $s_0=0$ and
$s_i=\delta_1+\cdots+\delta_i$. Select the five Hankel slices
with indices $s_0,\ldots,s_4$, allowing repeated selections as
in \cref{prop:border-certificate}. We prove the exact identity
\begin{equation}\label{eq:five-slice-rank}
\rank\matr{K}_{\CC^5}^{2}(\widehat{\tens{H}})
=\sum_{1\le i<j\le4}
\min\{\rho,\ell+\delta_i+\delta_j+1,
                 2\ell-\delta_i-\delta_j+1\}.
\end{equation} 
Put $f_i=x^{\ell-s_i}y^{s_i}$ for $0\le i\le4$ and define
the polynomial column vector $\vect{f}=(f_0,\ldots,f_4)^{\mathsf T}$.
Let $\matr{Q}$ be the $5\times4$ polynomial matrix whose $i$th column is
$y^{\delta_i}\vect{e}_{i-1}-x^{\delta_i}\vect{e}_i$.
Then $\matr{Q}^{\mathsf T}\vect{f}=0$, and its signed maximal minors
are $f_0,\ldots,f_4$. Since these have no common zero on
$\mathbb P^1$, they give the exact sequence
\[
0\to\mathcal O(-\ell)\to\mathcal O^5\to\set{E}\to0,
\] 
where $\set{E}=\bigoplus_{i=1}^4\mathcal O(\delta_i)$ is a vector bundle.
Taking the second exterior power and twisting by
$\mathcal O(\ell)$ gives an exact sequence with kernel $\set{E}$
and quotient $\bigoplus_{i<j}\mathcal O(\ell+\delta_i+\delta_j)$.
Here $\mathcal O(j)$ is the degree-$j$ line bundle on $\mathbb P^1$.
Because every $\delta_i\ge0$, one has $H^1(\mathbb P^1,\set{E})=0$.
Thus the coefficient map
\[
\matr{\Psi}=(\bigwedge\nolimits^2\matr{Q})^{\mathsf T}:
\set{R}_\ell^{10}\to\bigoplus_{i<j}\set{R}_{\ell+\delta_i+\delta_j}
\] 
is surjective.

Let $\matr{J}$ be the complementary-pair matrix on
$\bigwedge\nolimits^2\CC^4$: its $(\set{I},\set{I}^c)$ entry is the sign of the
concatenation $(\set{I},\set{I}^c)$, and its other entries are zero.
The exterior minor identity gives
$\matr{B}_{\vect{f}}=(\bigwedge\nolimits^2\matr{Q})\matr{J}
(\bigwedge\nolimits^2\matr{Q})^{\mathsf T}$, where $\matr{B}_{\vect{f}}$ represents
$(\vect{u},\vect{v})\mapsto\operatorname{vol}(\vect{f}\wedge \vect{u}\wedge \vect{v})$ on
$\bigwedge\nolimits^2\CC^5$.
To check the identity, intersecting row pairs give zero on both
sides; disjoint pairs give the corresponding signed maximal
minor of $\matr{Q}$.
Applying $L$ to its entries shows that the Hankel Koszul matrix
represents \[
(\vect{u},\vect{v})\mapsto L((\matr{\Psi} \vect{u})^{\mathsf T}
\matr{J}(\matr{\Psi} \vect{v}))
\] on $\set{R}_\ell^{10}$.
Surjectivity of $\matr{\Psi}$ preserves the rank of this pairing.
Each pair $\{i,j\}$ couples only to its complement, with polynomial
degrees $\ell+\delta_i+\delta_j$ and
$2\ell-\delta_i-\delta_j$. The preceding catalecticant rank
formula therefore proves \eqref{eq:five-slice-rank}.

Write $\ell=4k+\varepsilon$ with $0\le \varepsilon\le3$.
For $\varepsilon=0,1,3$, all six terms in \eqref{eq:five-slice-rank}
equal $\rho$. For $\varepsilon=2$, the pair sums have value $2k$ once,
value $2k+1$ four times, and value $2k+2$ once.
The apolar-degree bound gives $\rho\le\lfloor3\ell/2\rfloor+1=6k+4$
in this case. Thus four terms equal $\rho$ and two equal
$\min\{\rho,6k+3\}$. Only when $\rho=6k+4$ do two terms equal
$\rho-1$; otherwise all equal $\rho$.
Thus the rank is always at least $6\rho-2$.
Finally, identify the compressed first factor with the first
factor of $\tens{T}$ and retain the other factors in separate
blocks. As in \cref{lem:joint-koszul}, the resulting Koszul map
is the direct sum of the two maps. Its rank is at least
$6\rho-2+\rank\matr{K}_{\CC^5}^{2}(\tens{T})$.
Dividing by six and rounding up proves the assertion.
\end{proof}

\section{Symmetric third-order partners in at most five variables}
\label{sec:cubic-symmetric}

We now establish the partner bounds needed for the second condition
in \cref{thm:cubic-partners}. Let $\set{V}$ be the dual of the essential support of a symmetric
third-order tensor $\tens{S}$, so that $\tens{S}\in\Sym^3(\set{V}^*)$.
We associate to it the cubic form
\[
F_S(\vect{x})=\tens{S}(\vect{x},\vect{x},\vect{x}).
\]
The polar at $\vect{e}$ is the bilinear form
$\tens{S}(\vect{e},\cdot,\cdot)$, whose matrix is
$\frac16\operatorname{Hess}(F_S)(\vect{e})$.
We use $\br(F_S)$ and $\bs(F_S)$ for the corresponding ranks
of $\tens{S}$.
The main step is to show that a five-variable cubic whose
three-slice bound is at most six also has symmetric border rank
at most six, provided that it has an invertible polar.

\subsection{Cubic forms from small algebras}
We use Frobenius functionals in the sense of \cref{sec:moment}.
For a functional $\lambda$ on a commutative unital algebra,
write $\beta_\lambda(\vect{u},\vect{v})=\lambda(\vect{u} \vect{v})$.

\begin{lemma}[Small algebra bounds]\label{lem:small-algebra}
Let $\set{A}$ be a finite-dimensional commutative unital algebra
over $\CC$, and let $\lambda\in\set{A}^*$.
The following statements hold.
\begin{enumerate}
\item If $\dim\set{A}=t\le6$ and $\lambda$ is Frobenius,
the cubic $\vect{u}\mapsto\lambda(\vect{u}^3)$ has ordinary and symmetric
border rank $t$.
\item Suppose $\dim\set{A}=4$ and $\beta_\lambda$ has a
one-dimensional radical $\set{I}$. If $\eta,\mu\in\set{A}^*$
and $\eta|_{\set{I}}\ne0$, then, for every $c\in\CC$, the cubic
\begin{equation}\label{eq:corank-cubic}
F(\vect{u},z)=\lambda(\vect{u}^3)+3z\eta(\vect{u}^2)+3z^2\mu(\vect{u})+cz^3
\end{equation}
has symmetric border rank at most six.
\end{enumerate}
\end{lemma}
\begin{proof}
For (i), a Frobenius algebra is Gorenstein. In these lengths,
the Gorenstein locus is irreducible by
\cite[Theorem~A]{casnati2015gorenstein}. Its nonempty open locus of
reduced schemes is therefore dense, so the algebra is smoothable.
Trivialize its underlying vector space in a smoothing and extend
$\lambda$ linearly. On a reduced fiber the cubic is a weighted
sum of $t$ coordinate cubes, giving $\bs\le t$ in the limit.
The polar at the unit is $\beta_\lambda$, so ordinary matrix
flattening gives $\br\ge t$.

For (ii), the radical $\set{I}$ is an ideal on which $\lambda$
vanishes, and the induced pairing on $\set{A}/\set{I}$ is Frobenius.
First take $\set{A}=\CC^4$, with the first primitive idempotent
spanning $\set{I}$. In these coordinates, the cubic has the form
\[
F=\sum_{i=1}^3\lambda_i x_i^3
+3z\sum_{i=0}^3\eta_i x_i^2
+3z^2\sum_{i=0}^3\mu_i x_i+cz^3,
\] 
where $\lambda_1\lambda_2\lambda_3\eta_0\ne0$.
Set $y_i=x_i+(\eta_i/\lambda_i)z$ for $1\le i\le3$ and
$w=\sqrt{\eta_0}(x_0+\mu_0z/(2\eta_0))$.
Completing the three cubes writes
$F=\sum_{i=1}^3\lambda_i y_i^3+3zw^2+3z^2L$
for a linear form $L$.
The remainder is the limit, as $t\to0$, of
\[
((z+2tw+4t^2L)^3-2(z+tw+t^2L)^3+z^3)/(2t^2).
\]
It therefore has symmetric border rank at most three, proving
the bound for the reduced algebra.

To pass to a general algebra, deform the pair $(\set{A},\set{I})$.
Exactly one local factor $\set{A}_*$ contains $\set{I}$; put
$h=\dim\set{A}_*$. The other factors are Gorenstein of length
at most three, hence are of the form $\CC[x]/(x^j)$.
If $h=1$, the distinguished factor is already reduced; if $h=2$,
it is $\CC[x]/(x^2)$ with $\set{I}=(x)$.
For $h=3,4$, the Gorenstein quotient is
$\set{A}_*/\set{I}\simeq\CC[x]/(x^{h-1})$.
Choose a lift $x$ in the maximal ideal and a generator $\vect{v}$ of
$\set{I}$. The maximal ideal annihilates its one-dimensional
ideal $\set{I}$, since its action there is nilpotent.
Thus $1,x,\ldots,x^{h-2},\vect{v}$ is a basis, with
$x^{h-1}=\theta\vect{v}$ and $x\vect{v}=\vect{v}^2=0$ for some $\theta\in\CC$.
These relations determine an associative multiplication for every
$\theta$. If $\theta=0$, deform it to a nonzero value while keeping the
functionals fixed. The ideal and its Frobenius quotient pairing
are unchanged. For $\theta\ne0$, the factor is $\CC[x]/(x^h)$.

It remains to split the roots while retaining the distinguished
ideal. Choose distinct $a_1,\ldots,a_h$ and set
\[
p_t(x)=\prod_{j=1}^h(x-ta_j),\ \ q_t(x)=\prod_{j=1}^{h-1}(x-ta_j).
\]
In $\set{A}_t=\CC[x]/(p_t)$, the ideal $\set{I}_t=(q_t)$ has
dimension one and the quotient is $\CC[x]/(q_t)$.
For $h\ge2$, use the basis
$1,x,\ldots,x^{h-2},q_t/\theta$, taking $\theta=1$ when $h=2$.
Keep the values of $\lambda,\eta,\mu$ fixed in this basis.
In particular, the extended functionals satisfy
$\lambda_t(q_t)=0$ and $\eta_t(\set{I}_t)\ne0$.
The quotient pairing remains nondegenerate for general $t$
because it is nondegenerate at $t=0$.
Split the roots in the other local factors in the same way.
The general algebra is reduced, and its cubic has symmetric
border rank at most six by the preceding calculation.
Closedness of this rank locus, first under root splitting and
then under the deformation of $\theta$, proves (ii).
\end{proof}

The following observation will be used in several branches of
the argument. 
It also identifies the associated Frobenius algebra. 
\begin{lemma}[Commuting polars]\label{lem:commuting-polars}
Let $\tens{S}$ be a concise symmetric cubic tensor on a space
of dimension $d\le6$, and suppose its polar at $\vect{e}$ is invertible.
Define the bilinear form $\langle \vect{x},\vect{y}\rangle=\tens{S}(\vect{e},\vect{x},\vect{y})$
and the operators $\matr{L}_{\vect{x}}$ by
$\langle\matr{L}_{\vect{x}} \vect{y},\vect{z}\rangle=\tens{S}(\vect{x},\vect{y},\vect{z})$.
If these operators commute, then $\br(\tens{S})=\bs(\tens{S})=d$.
\end{lemma}
\begin{proof}
One has $\matr{L}_{\vect{e}}=\matr{I}$ and $\matr{L}_{\vect{x}} \vect{e}=\vect{x}$, and each
$\matr{L}_{\vect{x}}$ is self-adjoint for the nondegenerate bilinear form.
The unital algebra $\set{A}$ generated by the operators is
commutative, and $\vect{e}$ is cyclic. Thus the orbit map $a\mapsto a\vect{e}$
is an isomorphism from $\set{A}$ to the input space: if $a\vect{e}=0$,
commutativity gives $a(b\vect{e})=0$ for every $b\in\set{A}$.
Under this identification, the functional
$\lambda(a)=\langle \vect{e},a\vect{e}\rangle$ is Frobenius, since $\lambda(ab)=\langle a\vect{e},b\vect{e}\rangle$.
The cubic is $\lambda(a^3)$, so \cref{lem:small-algebra}(i)
gives the result.
\end{proof}

\subsection{The five-variable commuting extension}
\begin{proposition}[A six-cube border bound]\label{prop:six-cubes}
Let $\tens{S}\in\Sym^3(\set{V}^*)$ be concise, with $\dim\set{V}=5$.
Suppose that $\tens{S}$ has an invertible polar.
If $\kappa(\tens{S})\le6$, then $\bs(\tens{S})\le6$.
\end{proposition}
\begin{proof}
Choose a vector $\vect{e}$ whose polar is invertible and which satisfies $\tens{S}(\vect{e},\vect{e},\vect{e})\ne0$;
these two nonempty open conditions can be imposed simultaneously.
Use the bilinear form and the operators $\matr{L}_{\vect{x}}$
defined in \cref{lem:commuting-polars}.
After rescaling and changing coordinates, its bilinear form is
standard and $\langle \vect{e},\vect{e}\rangle=1$.
Write $\set{V}=\CC \vect{e}\oplus\set{U}$, with $\set{U}=\vect{e}^\perp$.
Then the following identities hold:
\[
\matr{L}_{\vect{e}}=\matr{I},\ \ 
\matr{L}_{\vect{x}} \vect{e}=\vect{x},\ \ 
\matr{L}_{\vect{x}}^{\mathsf T}=\matr{L}_{\vect{x}},\ \ 
\matr{L}_{\vect{x}} \vect{y}=\matr{L}_{\vect{y}} \vect{x}.
\]
All transposes refer to this complex bilinear form, not to a
Hermitian form.
For an orthonormal basis $\vect{e}_1,\ldots,\vect{e}_4$ of $\set{U}$, set
$\matr{L}_i:=\matr{L}_{\vect{e}_i}$ and write
$\matr{L}_i=\left(\begin{smallmatrix}0&\vect{e}_i^{\mathsf T}\\
\vect{e}_i&\matr{N}_i\end{smallmatrix}\right)$.
The matrices $\matr{N}_i$ are symmetric and satisfy
$\matr{N}_i\vect{e}_j=\matr{N}_j\vect{e}_i$.
Each commutator kills $\vect{e}$ and restricts on $\set{U}$ to
$\matr{C}_{ij}=\vect{e}_i\vect{e}_j^{\mathsf T}-\vect{e}_j\vect{e}_i^{\mathsf T}
+[\matr{N}_i,\matr{N}_j]$.
Identify $\vect{u}\wedge \vect{v}$ with $\vect{u} \vect{v}^{\mathsf T}-\vect{v} \vect{u}^{\mathsf T}$ and
let $\matr{C}:\bigwedge\nolimits^2\set{U}\to\bigwedge\nolimits^2\set{U}$ send
$\vect{e}_i\wedge \vect{e}_j$ to $\matr{C}_{ij}$.
Let $\matr{J}$ be the complementary-pair matrix in the ordered
basis $12,13,14,23,24,34$, whose nonzero upper-triangular entries are
\[
J_{12,34}=1,\ \  J_{13,24}=-1,\ \  J_{14,23}=1,
\]
and whose remaining nonzero entries are their transposes.
Put $n_{ijk}=(\matr{N}_i)_{jk}$.
The symmetry of $\matr{N}_i$ and the identities
$\matr{N}_i\vect{e}_j=\matr{N}_j\vect{e}_i$
make $n_{ijk}$ symmetric in all three indices.
Writing $C_{kl,ij}$ for the entry of $\matr{C}$ in row $kl$
and column $ij$, we obtain
$C_{kl,ij}=\delta_{ki}\delta_{lj}-\delta_{kj}\delta_{li}
+\sum_{a=1}^4(n_{ika}n_{jla}-n_{jka}n_{ila})$,
where $\delta$ is the Kronecker delta.
This formula gives $C_{kl,ij}=C_{ij,kl}$ and
$C_{12,34}-C_{13,24}+C_{14,23}=0$.
Consequently, $\matr{C}^{\mathsf T}=\matr{C}$ and
$\operatorname{tr}(\matr{J}\matr{C})
=2(C_{12,34}-C_{13,24}+C_{14,23})=0$.
These are the pair symmetry and the first Bianchi identity. 

Block elimination gives
$\rank\matr{K}(\matr{I},\matr{L}_{\vect{x}},\matr{L}_{\vect{y}})
=10+\rank[\matr{L}_{\vect{x}},\matr{L}_{\vect{y}}]$.
Hence the hypothesis bounds every commutator rank by two.
Equivalently, the map $\matr{C}$ sends decomposable bivectors to
decomposable bivectors or zero.
The decomposable locus is the irreducible quadric
$\vect{z}^{\mathsf T}\matr{J}\vect{z}=0$; therefore
$\matr{C}^{\mathsf T}\matr{J}\matr{C}=\gamma\matr{J}$
for some scalar $\gamma$.
If $\gamma\ne0$, the map is invertible. Otherwise its image
is totally isotropic for the wedge pairing and has dimension
at most three. We treat the resulting possibilities below.

\emph{A symmetric exterior factor.}
First suppose $\matr{C}=\bigwedge\nolimits^2\matr{G}$ for a symmetric
endomorphism $\matr{G}$ of $\set{U}$ with rank at least three.
Set $\vect{w}_i=\sqrt{-1}\,\matr{G}\vect{e}_i$, viewed as vectors in $\set{V}$.
Then \[
[\matr{L}_i,\matr{L}_j]+\vect{w}_i\vect{w}_j^{\mathsf T}
-\vect{w}_j\vect{w}_i^{\mathsf T}=0.
\]
Put $\vect{d}_{ij}=\matr{L}_i\vect{w}_j-\matr{L}_j\vect{w}_i$; symmetry of
$\matr{G}$ gives $\vect{d}_{ij}\in\set{U}$.
The Jacobi identity yields
\[
\vect{d}_{ij}\mathbin{\odot}\vect{w}_k+\vect{d}_{jk}\mathbin{\odot}\vect{w}_i
+\vect{d}_{ki}\mathbin{\odot}\vect{w}_j=0,
\] 
where
$\vect{u}\mathbin{\odot}\vect{v}=\vect{u} \vect{v}^{\mathsf T}+\vect{v} \vect{u}^{\mathsf T}$.
These equations give scalars $z_i$ with
$\vect{d}_{ij}=z_i\vect{w}_j-z_j\vect{w}_i$.
First suppose that $\vect{w}_1,\ldots,\vect{w}_4$ are independent.
For distinct $i,j,k$, the coefficient of
$\vect{w}_k\mathbin{\odot}\vect{w}_k$ in the Jacobi equation
forces the $\vect{w}_k$ component of $\vect{d}_{ij}$ to vanish.
Thus $\vect{d}_{ij}\in\Span(\vect{w}_i,\vect{w}_j)$.
Using $\vect{d}_{ji}=-\vect{d}_{ij}$, write
$\vect{d}_{ij}=a_{ij}\vect{w}_i-a_{ji}\vect{w}_j$ for $i\ne j$.
The coefficient of $\vect{w}_i\mathbin{\odot}\vect{w}_k$
then gives $a_{ij}=a_{kj}$ whenever $i,j,k$ are distinct.
Hence $a_{ij}$ depends only on $j$; setting $z_j=-a_{ij}$
gives the required formula.

If the $\vect{w}_i$ span a three-dimensional space, make an
invertible change in the index space so that
$\vect{w}_1,\vect{w}_2,\vect{w}_3$ are independent and
$\vect{w}_4=0$.
Both the Jacobi equations and the desired formula are
preserved under this change.
Complete $\vect{w}_1,\vect{w}_2,\vect{w}_3$ to a basis of
$\set{U}$ by a vector $\vect{u}$.
In the equation for $123$, the coefficients of
$\vect{u}\mathbin{\odot}\vect{w}_k$, for $k=1,2,3$,
first place $\vect{d}_{12},\vect{d}_{23},\vect{d}_{31}$
in $\Span(\vect{w}_1,\vect{w}_2,\vect{w}_3)$.
The square coefficients then place each vector in its
corresponding two-vector span.
The preceding coefficient comparison therefore gives
$z_1,z_2,z_3$ for these three pairs.
The equations for $ij4$ reduce to
$\vect{d}_{j4}\mathbin{\odot}\vect{w}_i
=\vect{d}_{i4}\mathbin{\odot}\vect{w}_j$.
Comparing coefficients shows that
$\vect{d}_{i4}=c\vect{w}_i$ for $i=1,2,3$, with one common
scalar $c$.
Set $z_4=-c$ and transform back to the original index basis. 
Consequently the symmetric matrices
$\widetilde{\matr{L}}_i=\left(\begin{smallmatrix}
\matr{L}_i&\vect{w}_i\\\vect{w}_i^{\mathsf T}&z_i
\end{smallmatrix}\right)$ commute on $\CC^6$.
Together with the identity they generate a commutative algebra
with cyclic vector $(\vect{e},0)^{\mathsf T}$: the first orbit vectors span $\set{V}$,
and some $\widetilde{\matr{L}}_i\widetilde{\matr{L}}_j(\vect{e},0)^{\mathsf T}$
has a nonzero last coordinate because $\matr{G}\ne0$.
As in \cref{lem:commuting-polars}, the orbit identification
supplies a six-dimensional Frobenius algebra.
Its cubic restricts to the original cubic on $\set{V}$.
By \cref{lem:small-algebra}(i), this proves $\bs(\tens{S})\le6$.

\emph{Invertible and three-dimensional images.}
If $\matr{C}$ is invertible, Chow's theorem gives either
$\matr{C}=\bigwedge\nolimits^2\matr{G}$ or
$\matr{C}=\matr{J}\bigwedge\nolimits^2\matr{G}$, after absorbing a scalar
into $\matr{G}$; see \cite[Theorem~3.6]{ghorpade2013grassmann}.
Here we use its complex-linear case for the Grassmannian of
two-planes in $\set{U}$: no nontrivial field automorphism occurs,
and every nonzero complex scalar has a square root.
In the first case, symmetry implies
$\matr{G}^{\mathsf T}=\matr{G}$ or
$\matr{G}^{\mathsf T}=-\matr{G}$.
The latter is impossible, since
$\operatorname{tr}(\matr{J}\bigwedge\nolimits^2\matr{G})
=-4\operatorname{Pf}(\matr{G})\ne0$ for an invertible skew matrix.
In the second case, symmetry and the complementary-minor identity
give $\bigwedge\nolimits^2(\matr{G}^2)=\det(\matr{G})\matr{I}$.
Thus $\matr{G}^2=\mu\matr{I}$ with $\mu^2=\det(\matr{G})\ne0$.
The eigenvalues are $\pm\sqrt\mu$, with even negative
multiplicity, so $\operatorname{tr}(\bigwedge\nolimits^2\matr{G})$ is
$6\mu$ or $-2\mu$. This contradicts
$\operatorname{tr}(\matr{J}\matr{C})=0$.
The symmetric exterior factor is therefore the only possibility.

If instead $\rank\matr{C}=3$ and its image is
$\bigwedge\nolimits^2\set{E}$ for a three-dimensional subspace
$\set{E}\subseteq\set{U}$, choose a full-column-rank matrix
$\matr{P}:\CC^3\to\set{U}$ with image $\set{E}$.
Symmetry gives
$\matr{C}=(\bigwedge\nolimits^2\matr{P})\matr{B}(\bigwedge\nolimits^2\matr{P})^{\mathsf T}$
for an invertible symmetric $3\times3$ matrix $\matr{B}$.
There is an invertible symmetric $\matr{G}_0$ with
$\bigwedge\nolimits^2\matr{G}_0=\matr{B}$: under complementary-pair
identification this is $\det(\matr{G}_0)\matr{G}_0^{-1}=\matr{B}'$
for a symmetric $\matr{B}'$, solved by
$\matr{G}_0=t(\matr{B}')^{-1}$ with $t^2=\det\matr{B}'$.
Then $\matr{G}=\matr{P}\matr{G}_0\matr{P}^{\mathsf T}$ is symmetric
of rank three and $\bigwedge\nolimits^2\matr{G}=\matr{C}$.
The preceding construction again applies.

\emph{A commuting hyperplane.}
If $\matr{C}=0$, \cref{lem:commuting-polars} gives
$\bs(\tens{S})=5$.
In every remaining case there is a nonzero $\vect{v}\in\set{U}$ with
$\operatorname{im}\matr{C}\subseteq \vect{v}\wedge\set{U}$.
Indeed, a linear space of decomposable bivectors lies in
$\vect{v}\wedge\set{U}$ or in $\bigwedge\nolimits^2\set{E}$ with $\dim\set{E}=3$
\cite[Lemma~3.5]{ghorpade2013grassmann}; a subspace of the
second type of dimension at most two also has a common factor.
Symmetry now implies that $\matr{C}$ kills
$\bigwedge\nolimits^2(\vect{v}^\perp\cap\set{U})$.
Thus the operators $\matr{L}_{\vect{x}}$ commute for $\vect{x}\in\set{W}=\vect{v}^\perp
\subseteq\set{V}$. Their unital algebra $\set{D}$ has orbit
$\set{D}\vect{e}$ containing $\set{W}$, so this orbit is $\set{V}$ or $\set{W}$.

If $\set{D}\vect{e}=\set{V}$, let $\matr{A}_{\vect{x}}\in\set{D}$ be the unique
operator with $\matr{A}_{\vect{x}} \vect{e}=\vect{x}$.
The orbit identification gives a five-dimensional Frobenius
cubic $F_A$. For $\vect{y}\in\set{W}$, one has
$\matr{A}_{\vect{y}}=\matr{L}_{\vect{y}}$ and
$\matr{A}_{\vect{x}} \vect{y}=\matr{L}_{\vect{y}} \vect{x}=\matr{L}_{\vect{x}} \vect{y}$.
Hence the symmetric operator $\matr{L}_{\vect{x}}-\matr{A}_{\vect{x}}$ kills
$\set{W}$. Full cubic symmetry implies that $F_S-F_A$ is a scalar
multiple of the cube of the linear form $\vect{x}\mapsto\langle \vect{v},\vect{x}\rangle$.
By \cref{lem:small-algebra}(i), we obtain $\bs(F_S)\le5+1=6$.

Suppose $\set{D}\vect{e}=\set{W}$.
If $\langle \vect{v},\vect{v}\rangle\ne0$, the orthogonal splitting
$\set{V}=\set{W}\oplus\CC \vect{v}$ makes all terms with one $\vect{v}$
and two inputs in $\set{W}$ vanish.
The restriction to $\set{W}$ has commuting polars and an
invertible polar at $\vect{e}$, so its symmetric border rank is four
by \cref{lem:commuting-polars}.
The remaining cubic is the square of one linear form times
another, whose symmetric border rank is at most two.
Thus again $\bs(F_S)\le6$.

Finally suppose $\langle \vect{v},\vect{v}\rangle=0$.
Give $\set{A}=\set{W}$ the product $\vect{a}\vect{b}=\matr{L}_{\vect{a}}\vect{b}$ and unit $\vect{e}$.
It is commutative and associative, since
$(\vect{a}\vect{b})\vect{c}=\matr{L}_{\vect{c}}\matr{L}_{\vect{a}}\vect{b}
=\matr{L}_{\vect{a}}\matr{L}_{\vect{c}}\vect{b}=\vect{a}(\vect{b}\vect{c})$.
The functional $\lambda(\vect{a})=\langle \vect{e},\vect{a}\rangle$ satisfies
$\lambda(\vect{a}\vect{b})=\langle \vect{a},\vect{b}\rangle$, whose radical on $\set{W}$
is $\CC \vect{v}$.
Choose $\vect{u}\notin\set{W}$ and define
\[
\eta(\vect{a})=\langle \vect{u},\vect{a}\rangle,\ \ 
\mu(\vect{a})=\tens{S}(\vect{u},\vect{u},\vect{a}),\ \ 
c=\tens{S}(\vect{u},\vect{u},\vect{u}).
\]
Then $\eta(\vect{v})\ne0$, and the cubic on $\set{W}\oplus\CC \vect{u}$
has exactly the form \eqref{eq:corank-cubic}.
Part (ii) of \cref{lem:small-algebra} gives the final bound.
\end{proof}

\subsection{Sharp partner bounds and the third main theorem}
\begin{proposition}[Symmetric partner bounds]\label{prop:symmetric-partner}
Let $\tens{S}$ be a symmetric cubic tensor with $d(\tens{S})\le5$.
Then $\br(\tens{S})=\bs(\tens{S})\le8$.
If this common value is at most seven, it equals $\kappa(\tens{S})$.
If it is eight, then $d(\tens{S})=5$ and
$\rank\matr{K}_{\CC^5}^{2}(\tens{S})=45$ on its essential space.
\end{proposition}
\begin{proof}
Pass to the essential space and write $d=d(\tens{S})$.
The zero tensor is immediate. The Alexander--Hirschowitz theorem
gives $\bs(\tens{S})\le d+1$ for $d\le4$, and
$\bs(\tens{S})\le8$ for $d=5$
\cite[Theorem~1.2]{brambilla2008alexander}.

First suppose there is an invertible polar $\matr{Q}$.
For any two other polars $\matr{P},\matr{R}$, block elimination gives
\[
\rank\matr{K}(\matr{Q},\matr{P},\matr{R})
=2d+\rank(\matr{P}\matr{Q}^{-1}\matr{R}
-\matr{R}\matr{Q}^{-1}\matr{P}).
\]
The last matrix is skew-symmetric, so a nonzero commutator gives
$\kappa(\tens{S})\ge d+1$.
If all commutators vanish, \cref{lem:commuting-polars} gives
$\br(\tens{S})=\bs(\tens{S})=d$, attained by the three slices
$(\matr{Q},0,0)$.
For $d\le4$, these alternatives and the preceding upper bound
prove all claims.
For $d=5$, each three-slice matrix is skew-symmetric of size
$15$, so $\kappa(\tens{S})\le7$.
If $\bs(\tens{S})\le7$ and the polars do not commute, then
$\kappa(\tens{S})\ge6$.
When $\kappa(\tens{S})=6$, \cref{prop:six-cubes} forces
$\bs(\tens{S})=6$; otherwise
$\kappa(\tens{S})=\br(\tens{S})=\bs(\tens{S})=7$.

The rank-eight case follows from Ottaviani's equation for the
seventh secant variety of the cubic Veronese variety in
$\mathbb P^4$. The full Koszul map vanishes on a fixed
five-dimensional subspace of its domain and takes values in a
fixed $45$-dimensional subspace of its target. Passing to the
corresponding quotient and subspace gives a $45\times45$ map
with the same rank as $\matr{K}_{\CC^5}^{2}(\tens{S})$.
Its determinant is a nonzero constant times the cube of the
defining equation
\cite[Theorem~1.1 and Lemma~3.1]{ottaviani2009invariant}.
Thus $\bs(\tens{S})=8$ gives Koszul rank $45$, and hence
$\br(\tens{S})\ge\lceil45/6\rceil=8$.

It remains to treat cubics with no invertible polar, equivalently
with identically vanishing Hessian. In at most four essential
variables no such nonzero concise cubic exists; see
\cite[Introduction]{gondim2015hessian}.
For five variables, first exclude a reducible concise cubic $\ell Q$.
Write $Q(\vect{x})=\vect{x}^{\mathsf T}\matr{A}\vect{x}$ with
$\matr{A}$ symmetric, and define $\rank Q=\rank\matr{A}$.
Conciseness forces $\rank Q\ge4$.
In rank four the cubic is equivalent to $x_0\sum_{i=1}^4x_i^2$,
which has nonzero Hessian determinant.
In rank five, write $\ell(\vect{x})=\vect{a}^{\mathsf T}\vect{x}$.
Direct differentiation gives
\[
\det\operatorname{Hess}(\ell Q)
=2^5\det(\matr{A})\ell^3
(4\ell^2-(\vect{a}^{\mathsf T}\matr{A}^{-1}\vect{a})Q),
\] 
which is nonzero.
Therefore the cubic is irreducible, and
\cite[Theorem~4.2]{gondim2015hessian} identifies it, up to
an invertible change of variables, with
$P=x_0u^2+2x_1uv+x_2v^2$.
Its symmetric border rank is five by
\cite[Proposition~7.5, with $n=4$]{huang2020wild}.
For an explicit three-slice bound, take the tensor of $3P$ and
the polars in the $x_0,u,v$ directions. With coordinates
$(x_0,x_1,x_2,u,v)$ numbered from zero, their matrices are
\[
\matr{M}_0=\matr{E}_{33},\ \ 
\matr{M}_1=\matr{E}_{03}+\matr{E}_{30}+\matr{E}_{14}+\matr{E}_{41},\ \ \matr{M}_2=\matr{E}_{13}+\matr{E}_{31}+\matr{E}_{24}+\matr{E}_{42},
\]
where $\matr{E}_{ij}$ denotes a matrix unit.
The kernel of \eqref{eq:koszul-matrix} consists of triples
\[
((a,0,0,0,0),(b,c,d,a,0),(c,d,f,0,a)).
\] 
It has dimension five, so the matrix rank is ten.
Thus $\kappa(\tens{S})=\br(\tens{S})=\bs(\tens{S})=5$,
which completes the remaining case.
\end{proof}

\begin{proof}[Proof of \cref{thm:cubic-partners}]
Under condition \textup{(i)}, \cref{prop:small-border-partner}
and the common matrix bound in \cref{sec:cubic-bounds} give
$\br(\tens{H}\oplus\tens{T})\ge\rho+\br(\tens{T})$.
Subadditivity gives equality.
Under condition \textup{(ii)}, \cref{prop:symmetric-partner}
gives $\br(\tens{T})=\bs(\tens{T})$.
If this value is at most seven, the same three-slice argument
applies. If it is eight, \cref{prop:five-slice} gives
\[
\br(\tens{H}\oplus\tens{T})\ge
\rho+\lceil(45-2)/6\rceil=\rho+8,
\]
again matching subadditivity.
Finally, if $\tens{T}$ in condition \textup{(i)} is symmetric,
then $d(\tens{T})\le\br(\tens{T})\le3$, so condition
\textup{(ii)} also holds.
In every symmetric case, \cref{cor:border-symmetric} now proves
the claimed symmetric equality.
\end{proof}

\section{Consequences and examples}
\label{sec:conclusion}

We first derive finite-sum consequences of the main results and
compare exact rank with border rank for complex Hankel tensors.
We then use a single family of examples to illustrate these
conclusions and the sharpness of the restriction bound. 

\subsection{Finite direct sums}
\begin{corollary}[Finite Hankel sums]
\label{cor:finite}
Let $\tens{H}_1,\ldots,\tens{H}_k$ be complex
Hankel tensors of the same order $m\ge3$,
and let $\tens{T}$ be any complex tensor of that order.
Then
\[
 r\left(\tens{T}\oplus\bigoplus_{\nu=1}^k\tens{H}_\nu\right)
 =r(\tens{T})+\sum_{\nu=1}^k r(\tens{H}_\nu).
\]
For a symmetric partner $\tens{S}$ with $\Delta_c(\tens{S})=0$,
the corresponding sum has symmetric rank
$s(\tens{S})+\sum_{\nu=1}^k s(\tens{H}_\nu)$.
For the Hankel summands alone, both border ranks satisfy
\begin{equation}\label{eq:finite-border}
\br\left(\bigoplus_{\nu=1}^k\tens{H}_\nu\right)
=\bs\left(\bigoplus_{\nu=1}^k\tens{H}_\nu\right)
=\sum_{\nu=1}^k\rho(\tens{H}_\nu).
\end{equation}
\end{corollary}
\begin{proof}
Apply \cref{thm:main} successively, treating
the sum of the remaining components as the
unrestricted partner.
In the symmetric case the CP lower bound equals
the sum of the symmetric ranks, and concatenation
gives the matching upper bound.
For border rank, \eqref{eq:border-upper} gives the upper bound.
The lower bound is the sum of the balanced flattening ranks
when $m$ is even, and follows from
\cref{lem:joint-koszul,prop:border-certificate} when $m$ is odd.
\end{proof}

The exact-rank assertion is obtained by iteration because
\cref{thm:main} permits an arbitrary partner. The border assertion
uses the joint matrix construction instead: an intermediate
direct sum need not itself be Hankel.
All CP-rank conclusions persist under invertible changes
in the individual modes. The symmetric conclusions persist
when the same invertible map is used in every mode.

\subsection{Exact rank versus border rank}
\begin{corollary}[The rank difference]\label{cor:rank-difference}
Let $\tens{H}\ne\tens{0}$ belong to $\set{H}_m(n)$ with $m\ge3$.
Use the apolar parameters $\rho,\sigma$ from \cref{sec:moment}
and put $D=m(n-1)$. Then
\begin{equation}\label{eq:rank-difference}
r(\tens{H})-\br(\tens{H})=s(\tens{H})-\bs(\tens{H})
=\min\{D-2\rho+2,\ (m-2)(\rho-\sigma)\}.
\end{equation}
Consequently, exact rank equals border rank if and only if
\begin{equation}\label{eq:rank-border-equality}
D=2\rho-2\ \text{or}\ \sigma=\rho.
\end{equation}
\end{corollary}
\begin{proof}
In the original dimension, \cref{prop:upper} reads
$$r(\tens{H})=s(\tens{H})
=\min\{D-\rho+2,(m-1)\rho-(m-2)\sigma\}.$$
Subtract the value $\rho$ supplied by
\cref{cor:hankel-border-ranks}.
Both entries of the resulting minimum are nonnegative,
and $m-2>0$, proving the equivalence.
The scalar convention in \cref{sec:moment} gives the same formulas.
\end{proof}

% \begin{remark}[Balanced degree and repeated roots]
% This is a consequence of the exact formula
% \cite[Theorem~1.1]{holden2026exact} and the border theorem,
% with no additional decomposition argument.
% The first condition in \eqref{eq:rank-border-equality} is the
% balanced apolar case; the second says that the chosen minimum
% apolar polynomial is squarefree. In the balanced case the choice
% of that polynomial need not be unique, but the first entry in
% \eqref{eq:rank-difference} is already zero, so the answer is
% independent of the choice.
% \end{remark}

\subsection{Examples from one moment family}
Let $m\ge3$ and $n\ge2$. For an integer
$1\le\rho\le\lfloor m(n-1)/2\rfloor+1$, define
\begin{equation}\label{eq:antidiagonal-family}
H^{(\rho)}_{i_1\ldots i_m}=
\begin{cases}
1,&i_1+\cdots+i_m=\rho-1,\\
0,&\text{otherwise},
\end{cases}
\ \  0\le i_1,\ldots,i_m<n.
\end{equation}
Its middle binary catalecticant has an antidiagonal identity
block of size $\rho$ and zeros elsewhere, so its rank is
$\rho$. Moreover, the polynomial $g_H=t^\rho$ is apolar and is therefore of
minimum degree. Thus $\sigma=1$ and $b=\min(n,\rho)$.
Put $D=m(n-1)$. Then \cref{prop:upper,cor:hankel-border-ranks} give
\begin{equation}\label{eq:antidiagonal-ranks}
r(\tens{H}^{(\rho)})=s(\tens{H}^{(\rho)})
=\min\{D-\rho+2,(m-1)\rho-(m-2)\},
\ \  \br(\tens{H}^{(\rho)})=\bs(\tens{H}^{(\rho)})=\rho.
\end{equation}
The following examples are direct specializations of this calculation.

\begin{example}[An unbounded rank difference]\label{ex:large-difference}
Set $m=3$ and $\rho=n$ in \eqref{eq:antidiagonal-family}, so
$H^{(n)}_{ijk}=1$ exactly when $i+j+k=n-1$.
By \eqref{eq:antidiagonal-ranks}, we have
$$r(\tens{H}^{(n)})=s(\tens{H}^{(n)})=2n-1,\ \ \br(\tens{H}^{(n)})=\bs(\tens{H}^{(n)})=n.$$ 
The difference is $n-1$, whereas \cref{thm:main} still gives
$\Delta_r(\tens{H}^{(n)},\tens{T})=0$ for every complex third-order tensor
$\tens{T}$. Thus the difference between exact rank and border
rank is distinct from the Strassen CP rank gap.
\end{example}

\begin{example}[The coefficient $m-1$ is sharp]\label{ex:sharp}
For $n=\rho=2$, the tensor in \eqref{eq:antidiagonal-family} is
\[
\tens{H}^{(2)}=\sum_{i=1}^m
\vect{e}_0^{\otimes(i-1)}\otimes\vect{e}_1
\otimes\vect{e}_0^{\otimes(m-i)}.
\]
Its moment parameters are $(b,\rho,a,\sigma)=(2,2,0,1)$, and
\eqref{eq:antidiagonal-ranks} gives
$r(\tens{H}^{(2)})=s(\tens{H}^{(2)})=m$ and
$\br(\tens{H}^{(2)})=\bs(\tens{H}^{(2)})=2$.
Restricting one input to $\CC\vect{e}_1^*$ leaves a tensor of
CP rank one. A codimension-one restriction therefore loses
exactly $m-1$ in CP rank, showing that the coefficient in
\cref{prop:restriction} cannot be decreased.
\end{example}

\begin{example}[Compressed nonreduced moment models]\label{ex:compressed}
Suppose $n<\rho\le\lfloor m(n-1)/2\rfloor+1$ in
\eqref{eq:antidiagonal-family}.
The input space $\CC[t]_{<n}$ is then a proper subspace of
$\CC[t]/(t^\rho)$.
The second entry of the minimum in \eqref{eq:antidiagonal-ranks}
exceeds the first by $m(\rho-n)$; hence
$$r(\tens{H}^{(\rho)})=s(\tens{H}^{(\rho)})=m(n-1)-\rho+2$$ 
and
$\br(\tens{H}^{(\rho)})=\bs(\tens{H}^{(\rho)})=\rho$.
Both direct-sum theorems therefore apply with a compressed
input space and a repeated apolar root.
For $m=5$ and $n=3$ with $\rho=6$, the usual $2\mid3$ flattening
has rank $5$, whereas the compressed Koszul matrix has rank
$12$ and certifies border rank $6$.
\end{example}

\section{Conclusions}
\label{sec:outlook}

We have established CP rank additivity for a complex Hankel tensor
and an arbitrary complex tensor of the same order, for all orders
at least three. We have also shown that ordinary and symmetric
border ranks coincide and are additive for direct sums of
complex Hankel tensors. In order three, border-rank additivity
also holds for arbitrary partners of border rank at most three
and for symmetric partners with at most five essential variables.
These sufficient conditions leave the case of an unrestricted
partner open. 
This leads to the following question,
which, as explained in \cref{subsec:bord_rank_part}, includes the problem
posed in \cite[Problem~2]{buczynski2020additivity} as a special case.

\begin{problem}\label{prob:border-additivity}
Let $m\ge3$ and $\tens{H}\in\set{H}_m(n)$.
Does $\br(\tens{H}\oplus\tens{T})
=\br(\tens{H})+\br(\tens{T})$ hold for every complex tensor
$\tens{T}$ of order $m$?
\end{problem}

\noindent\textbf{Acknowledgments.} 
The author acknowledges the assistance of ChatGPT 6 Astra throughout the mathematical development and preparation of this manuscript. 
The author takes full responsibility for the content of this paper. 
This work was supported by the Guangdong Basic and Applied Basic
Research Foundation (2026A1515050001, 2026A1515010188).

\bibliographystyle{siamplain}
\bibliography{ReferenceTensor}

\begin{thebibliography}{10}

\bibitem{barrilli2026eigensolvers}
{\sc E.~Barrilli and B.~Mourrain}, {\em Eigensolvers for polynomial roots and
  tensor decomposition}, 2026, \url{https://arxiv.org/abs/2608.19818v2}.
\newblock Preprint, arXiv:2608.19818v2.

\bibitem{bechere2026rational}
{\sc M.~Bechere, S.~Kuhlmann, and B.~Mourrain}, {\em Symmetric tensor
  decomposition on rational varieties}, 2026,
  \url{https://arxiv.org/abs/2606.25712v3}.
\newblock Preprint, arXiv:2606.25712v3.

\bibitem{beck2017additive}
{\sc V.~Beck and C.~Lecouvey}, {\em Additive combinatorics methods in
  associative algebras}, Confluentes Mathematici, 9 (2017), pp.~3--27.

\bibitem{bernardi2026hankel}
{\sc A.~Bernardi, J.~Jelisiejew, and O.~{Reig Fit{\'e}}}, {\em {Hankel} and
  multiplication tensor completions for cactus rank}, 2026,
  \url{https://arxiv.org/abs/2606.30600v1}.
\newblock Preprint, arXiv:2606.30600v1.

\bibitem{bini1979approximate}
{\sc D.~Bini, M.~Capovani, F.~Romani, and G.~Lotti}, {\em {$O(n^{2.7799})$}
  complexity for {$n\times n$} approximate matrix multiplication}, Information
  Processing Letters, 8 (1979), pp.~234--235.

\bibitem{bini1980approximate}
{\sc D.~Bini, G.~Lotti, and F.~Romani}, {\em Approximate solutions for the
  bilinear form computational problem}, SIAM Journal on Computing, 9 (1980),
  pp.~692--697.

\bibitem{blaser2023multilinear}
{\sc M.~Bl{\"a}ser, H.~Mayer, and D.~Shringi}, {\em On the multilinear
  complexity of associative algebras}, in 40th International Symposium on
  Theoretical Aspects of Computer Science (STACS 2023), P.~Berenbrink,
  P.~Bouyer, A.~Dawar, and M.~M. Kant{\'e}, eds., vol.~254 of Leibniz
  International Proceedings in Informatics (LIPIcs), Dagstuhl, Germany, 2023,
  Schloss Dagstuhl -- Leibniz-Zentrum f{\"u}r Informatik, pp.~12:1--12:18,
  \url{https://doi.org/10.4230/LIPIcs.STACS.2023.12}.

\bibitem{brachat2010symmetric}
{\sc J.~Brachat, P.~Comon, B.~Mourrain, and E.~Tsigaridas}, {\em Symmetric
  tensor decomposition}, Linear Algebra and its Applications, 433 (2010),
  pp.~1851--1872.

\bibitem{brambilla2008alexander}
{\sc M.~C. Brambilla and G.~Ottaviani}, {\em On the {Alexander-Hirschowitz}
  theorem}, Journal of Pure and Applied Algebra, 212 (2008), pp.~1229--1251.

\bibitem{buczynski2014third}
{\sc J.~Buczy{\'n}ski and J.~M. Landsberg}, {\em On the third secant variety},
  Journal of Algebraic Combinatorics, 40 (2014), pp.~475--502.

\bibitem{buczynski2020additivity}
{\sc J.~Buczy{\'n}ski, E.~Postinghel, and F.~Rupniewski}, {\em On {Strassen}'s
  rank additivity for small three-way tensors}, SIAM Journal on Matrix Analysis
  and Applications, 41 (2020), pp.~106--133.

\bibitem{carlini2015symmetricstrassen}
{\sc E.~Carlini, M.~V. Catalisano, and L.~Chiantini}, {\em Progress on the
  symmetric {Strassen} conjecture}, Journal of Pure and Applied Algebra, 219
  (2015), pp.~3149--3157.

\bibitem{carlini2018ecomputability}
{\sc E.~Carlini, M.~V. Catalisano, L.~Chiantini, A.~V. Geramita, and Y.~Woo},
  {\em Symmetric tensors: rank, {Strassen}'s conjecture and
  {$e$}-computability}, Annali della Scuola Normale Superiore di Pisa, Classe
  di Scienze (5), 18 (2018), pp.~363--390,
  \url{https://doi.org/10.2422/2036-2145.201602_004}.

\bibitem{carlini2012monomials}
{\sc E.~Carlini, M.~V. Catalisano, and A.~V. Geramita}, {\em The solution to
  the {Waring} problem for monomials and the sum of coprime monomials}, Journal
  of Algebra, 370 (2012), pp.~5--14.

\bibitem{carlini2017waringloci}
{\sc E.~Carlini, M.~V. Catalisano, and A.~Oneto}, {\em {Waring} loci and the
  {Strassen} conjecture}, Advances in Mathematics, 314 (2017), pp.~630--662.

\bibitem{casarotti2020comon}
{\sc A.~Casarotti, A.~Massarenti, and M.~Mella}, {\em On {Comon}'s and
  {Strassen}'s conjectures}, Mathematics, 6 (2018),
  \url{https://doi.org/10.3390/math6110217}.
\newblock Article 217.

\bibitem{casnati2015gorenstein}
{\sc G.~Casnati, J.~Jelisiejew, and R.~Notari}, {\em Irreducibility of the
  {Gorenstein} loci of {Hilbert} schemes via ray families}, Algebra \& Number
  Theory, 9 (2015), pp.~1525--1570.

\bibitem{christandl2021border}
{\sc M.~Christandl, F.~Gesmundo, M.~Micha{\l}ek, and J.~Zuiddam}, {\em Border
  rank nonadditivity for higher order tensors}, SIAM Journal on Matrix Analysis
  and Applications, 42 (2021), pp.~503--527.

\bibitem{colbrook2026generic}
{\sc M.~J. Colbrook}, {\em Generic ranks of odd-order {Hankel} tensors from
  three compressed slices}, 2026,
  \url{https://github.com/ajt60gaibb/OpenProblemsInNLA/blob/e62f9db721430d40c94e103249bc79586d1d97b3/references/colbrook-recovered-tensors-2026-09-11/manuscripts/TR-13.tex}.
\newblock Repository manuscript, TR-13, Open Problems in Numerical Linear
  Algebra.

\bibitem{comon2008symmetric}
{\sc P.~Comon, G.~Golub, L.-H. Lim, and B.~Mourrain}, {\em Symmetric tensors
  and symmetric tensor rank}, SIAM Journal on Matrix Analysis and Applications,
  30 (2008), pp.~1254--1279.

\bibitem{de2008tensor}
{\sc V.~De~Silva and L.-H. Lim}, {\em Tensor rank and the ill-posedness of the
  best low-rank approximation problem}, SIAM Journal on Matrix Analysis and
  Applications, 30 (2008), pp.~1084--1127.

\bibitem{ding2015fast}
{\sc W.~Ding, L.~Qi, and Y.~Wei}, {\em Fast {Hankel} tensor--vector product and
  its application to exponential data fitting}, Numerical Linear Algebra with
  Applications, 22 (2015), pp.~814--832.

\bibitem{ding2017inheritance}
{\sc W.~Ding, L.~Qi, and Y.~Wei}, {\em Inheritance properties and
  sum-of-squares decomposition of {Hankel} tensors: Theory and algorithms}, BIT
  Numerical Mathematics, 57 (2017), pp.~169--190.

\bibitem{edmonds1968partition}
{\sc J.~Edmonds}, {\em Matroid partition}, in 50 Years of Integer Programming
  1958--2008: From the Early Years to the State-of-the-Art, Springer, Berlin,
  Heidelberg, 2010, pp.~199--217,
  \url{https://doi.org/10.1007/978-3-540-68279-0_7}.

\bibitem{feig1984directsum}
{\sc E.~Feig and S.~Winograd}, {\em On the direct sum conjecture}, Linear
  Algebra and its Applications, 63 (1984), pp.~193--219.

\bibitem{ghorpade2013grassmann}
{\sc S.~R. Ghorpade and K.~V. Kaipa}, {\em Automorphism groups of {Grassmann}
  codes}, Finite Fields and Their Applications, 23 (2013), pp.~80--102.

\bibitem{goemans2009union}
{\sc M.~X. Goemans}, {\em Lecture 12: Matroid intersection polytope and matroid
  union}.
\newblock MIT 18.438, Advanced Combinatorial Optimization, 2009,
  \url{https://math.mit.edu/~goemans/18438F09/lec12.pdf}.
\newblock October 27. Scribe: Chung Chan.

\bibitem{gondim2015hessian}
{\sc R.~Gondim and F.~Russo}, {\em On cubic hypersurfaces with vanishing
  {Hessian}}, Journal of Pure and Applied Algebra, 219 (2015), pp.~779--806.

\bibitem{hitchcock1927expression}
{\sc F.~L. Hitchcock}, {\em The expression of a tensor or a polyadic as a sum
  of products}, Journal of Mathematics and Physics, 6 (1927), pp.~164--189.

\bibitem{holden2026exact}
{\sc S.~Holden}, {\em Exact rank of complex {Hankel} tensors}, 2026,
  \url{https://github.com/ajt60gaibb/OpenProblemsInNLA/blob/5d79588e225376c81f1a664f3ec128a482bf7103/tensor-computations/TR-14/solution.tex}.
\newblock Repository manuscript, TR-14, Open Problems in Numerical Linear
  Algebra.

\bibitem{huang2020wild}
{\sc H.~Huang, M.~Micha{\l}ek, and E.~Ventura}, {\em Vanishing {Hessian}, wild
  forms and their border {VSP}}, Mathematische Annalen, 378 (2020),
  pp.~1505--1532, \url{https://doi.org/10.1007/s00208-020-02080-8}.

\bibitem{jaja1986validity}
{\sc J.~Ja’Ja’ and J.~Takche}, {\em On the validity of the direct sum
  conjecture}, SIAM Journal on Computing, 15 (1986), pp.~1004--1020.

\bibitem{kolda2009tensor}
{\sc T.~G. Kolda and B.~W. Bader}, {\em Tensor decompositions and
  applications}, SIAM {Review}, 51 (2009), pp.~455--500.

\bibitem{kruskal1977three}
{\sc J.~B. Kruskal}, {\em Three-way arrays: rank and uniqueness of trilinear
  decompositions, with application to arithmetic complexity and statistics},
  Linear Algebra and its Applications, 18 (1977), pp.~95--138.

\bibitem{landsberg2012tensors}
{\sc J.~M. Landsberg}, {\em Tensors: Geometry and Applications}, vol.~128 of
  Graduate Studies in Mathematics, American Mathematical Society, 2012.

\bibitem{landsberg2015koszul}
{\sc J.~M. Landsberg and G.~Ottaviani}, {\em New lower bounds for the border
  rank of matrix multiplication}, Theory of Computing, 11 (2015), pp.~285--298,
  \url{https://doi.org/10.4086/toc.2015.v011a011}.

\bibitem{li2026comonrankgap}
{\sc J.~Li}, {\em Upper bounds on the {Comon} rank gap of real symmetric
  third-order tensors}, arXiv preprint arXiv:2610.14002,  (2026),
  \url{https://arxiv.org/abs/2610.14002}.

\bibitem{luque2003hankel}
{\sc J.-G. Luque and J.-Y. Thibon}, {\em {Hankel} hyperdeterminants and
  {Selberg} integrals}, Journal of Physics A: Mathematical and General, 36
  (2003), pp.~5267--5292.

\bibitem{mandziuk2026symmetrization}
{\sc T.~Ma{\'n}dziuk and E.~Ventura}, {\em Symmetrization maps and minimal
  border rank {Comon}'s conjecture}, Selecta Mathematica, 32 (2026),
  \url{https://doi.org/10.1007/s00029-026-01182-6}.
\newblock Article 86.

\bibitem{mourrain2018moments}
{\sc B.~Mourrain}, {\em Polynomial--exponential decomposition from moments},
  Foundations of Computational Mathematics, 18 (2018), pp.~1435--1492.

\bibitem{nie2019hankel}
{\sc J.~Nie and K.~Ye}, {\em Hankel tensor decompositions and ranks}, SIAM
  Journal on Matrix Analysis and Applications, 40 (2019), pp.~486--516.

\bibitem{ottaviani2009invariant}
{\sc G.~Ottaviani}, {\em An invariant regarding {Waring}'s problem for cubic
  polynomials}, Nagoya Mathematical Journal, 193 (2009), pp.~95--110.

\bibitem{papy2005exponential}
{\sc J.-M. Papy, L.~De~Lathauwer, and S.~Van~Huffel}, {\em Exponential data
  fitting using multilinear algebra: the single-channel and multi-channel
  case}, Numerical Linear Algebra with Applications, 12 (2005), pp.~809--826.

\bibitem{qi2015hankel}
{\sc L.~Qi}, {\em {Hankel} tensors: Associated {Hankel} matrices and
  {Vandermonde} decomposition}, Communications in Mathematical Sciences, 13
  (2015), pp.~113--125.

\bibitem{rupniewski2024additivity}
{\sc F.~Rupniewski}, {\em Strassen's rank additivity for small tensors,
  including tensors of rank less or equal 7}, Linear Algebra and its
  Applications, 698 (2024), pp.~448--478.

\bibitem{schonhage1981partial}
{\sc A.~Sch{\"o}nhage}, {\em Partial and total matrix multiplication}, SIAM
  Journal on Computing, 10 (1981), pp.~434--455.

\bibitem{seigal2020ranks}
{\sc A.~Seigal}, {\em Ranks and symmetric ranks of cubic surfaces}, Journal of
  Symbolic Computation, 101 (2020), pp.~304--317.

\bibitem{shitov2019directsum}
{\sc Y.~Shitov}, {\em Counterexamples to {Strassen}'s direct sum conjecture},
  Acta Mathematica, 222 (2019), pp.~363--379.

\bibitem{shitov2023higher}
{\sc Y.~Shitov}, {\em Higher rank substitutions for tensor decompositions. {I}.
  {Direct} sum conjectures}, 2023, \url{https://rxiv.org/pdf/2307.0100v1.pdf}.
\newblock Preprint, viXra:2307.0100v1.

\bibitem{strassen1973division}
{\sc V.~Strassen}, {\em Vermeidung von {Divisionen}}, Journal f{\"u}r die reine
  und angewandte Mathematik, 264 (1973), pp.~184--202.

\bibitem{teitler2015sufficient}
{\sc Z.~Teitler}, {\em Sufficient conditions for {Strassen}'s additivity
  conjecture}, Illinois Journal of Mathematics, 59 (2015), pp.~1071--1085.

\bibitem{xu2016hankel}
{\sc C.~Xu}, {\em {Hankel} tensors, {Vandermonde} tensors and their
  positivities}, Linear Algebra and its Applications, 491 (2016), pp.~56--72,
  \url{https://doi.org/10.1016/j.laa.2015.02.012}.

\end{thebibliography}
\end{document}